\documentclass[11pt,twoside]{article}

\usepackage{mathrsfs}
\usepackage{amssymb}
\usepackage{amsmath,mathrsfs}
\usepackage{amsthm}
\usepackage{amsfonts}
\usepackage{color}
\usepackage{latexsym}
\usepackage{txfonts}
\usepackage{indentfirst}
\usepackage{soul}
\usepackage{enumerate}
\usepackage{appendix}

\usepackage{anysize}

\allowdisplaybreaks

\newtheorem{theorem}{Theorem}[section]
\newtheorem{lemma}[theorem]{Lemma}
\newtheorem{corollary}[theorem]{Corollary}
\newtheorem{proposition}[theorem]{Proposition}
\theoremstyle{definition}
\newtheorem{remark}[theorem]{Remark}
\newcounter{assum}

\newtheorem{assumption}[assum]{Assumption}
\newtheorem{definition}[theorem]{Definition}

\numberwithin{equation}{section}

\begin{document}
\title{\Large\bf Robin Problems of Elliptic Equations on Rough Domains:
H\"older Regularity, Green's Functions, and Harmonic Measures
\footnotetext{\hspace{-0.35cm} 2020 {\it Mathematics Subject Classification}.
{Primary 35J25; Secondary 35J15, 35J08, 35B65, 42B37.}
\endgraf{\it Key words and phrases}. Robin problem,
elliptic operator, rough domain, Harnack inequality, harmonic measure, Green's function.
\endgraf This project is partially supported by the National
Natural Science Foundation of China (Grant Nos. 12431006 and 12371093),
the Key Project of Gansu Provincial National Science Foundation (Grant No.
23JRRA1022), Longyuan Young Talents of Gansu Province,
and the Fundamental Research Funds for the Central Universities (Grant No. 2253200028).}}
\author{Jiayi Wang, Dachun Yang\footnote{Corresponding author, E-mail:
\texttt{dcyang@bnu.edu.cn}/{\color{red}\today}/Final version.},\ \ and Sibei Yang}
\date{}
\maketitle
	
\vspace{-0.8cm}
	
\begin{center}
\begin{minipage}{13.5cm}\small
{{\bf Abstract.} Let $n\ge 2$ and $s\in (n-2,n)$. Assume that $\Omega\subset
\mathbb{R}^n$ is a one-sided bounded non-tangentially accessible
 domain with $s$-Ahlfors regular boundary and $\sigma$ is
the surface measure on the boundary of $\Omega$, denoted by $\partial \Omega$.
Let $\beta$ be a non-negative measurable function on $\partial \Omega$
satisfying
$
\beta\in L^{q_0}(\partial \Omega,\sigma)~\text{with}~
q_0 \in(\frac{s}{s+2-n},\infty]~\text{and}\
\beta\ge a_0~\text{on}~E_0\subset \partial \Omega,
$
where $a_0$ is a given positive constant and $E_0\subset \partial \Omega$ is
a $\sigma$-measurable set with $\sigma(E_0)>0$. In this article,
for any $f\in L^p(\partial \Omega,\sigma)$ with $p\in(s/(s+2-n),\infty]$,
we obtain the existence and uniqueness, the global H\"older regularity, and the boundary Harnack inequality
of the weak solution to the Robin problem
$$\begin{cases}
-\mathrm{div}(A\nabla u) = 0~~&\text{in}~\Omega,\\
A\nabla u\cdot \boldsymbol{\nu}+\beta u = f~~&\text{on}~\partial \Omega,
\end{cases}
$$
where the coefficient matrix $A$ is real-valued, bounded and measurable
and satisfies the uniform ellipticity condition and where $\boldsymbol{\nu}$
denotes the outward unit normal to $\partial\Omega$. Furthermore, we
establish the existence, upper bound pointwise estimates, and the H\"older regularity
of Green's functions associated with this Robin problem.
As applications, we further prove that the harmonic measure associated with
this Robin problem is mutually absolutely continuous with respect to
the surface measure $\sigma$ and also provide a quantitative characterization
of mutual absolute continuity at small scales. These results extend the corresponding
results established by David et al. [arXiv: 2410.23914] via weakening their assumption that $\beta$ is a given positive constant.}
\end{minipage}
\end{center}
	
\vspace{0.1cm}
	
\section{Introduction and main results \label{S1}}
	
Let $n\ge 2$, $s\in (n-2,n)$, and $\Omega\subset\mathbb{R}^n$ be a one-sided bounded non-tangentially accessible
(for short, NTA) domain with mixed $s$-Ahlfors regular boundary [see Assumption
\ref{Ass} for the details]. Denote by $\sigma$ the surface measure on the boundary of $\Omega$, denoted by $\partial \Omega$.
Assume that $\beta$ is a non-negative measurable function on $\partial \Omega$ satisfying
\begin{equation}\label{eq1.1}
\beta\ge a_0~\text{on}~E_0\subset \partial \Omega~\text{and}\ \beta\in L^{q_0}(\partial \Omega,\sigma)~\text{with}~
\begin{cases}
q_0 \in \left[\dfrac{s}{s+2-n},\infty\right]&\text{when}~n\ge 3,\\
q_0 \in (1,\infty]&\text{when}~n= 2,
\end{cases}
\end{equation}
where $a_0\in (0,\infty)$ is a given constant and $E_0\subset 
\partial \Omega$ is a $\sigma$-measurable subset satisfying $\sigma(E_0)>0$.
Let $f\in L^p(\partial \Omega,\sigma)$ with some $p\in(s/(s+2-n),\infty]$.
In this article, motivated by the recent work of David et al. \cite{DDEMM24,DDEFMM} on the harmonic measure associated
with the Robin boundary value problem of the elliptic operator $-\mathrm{div}(A\nabla\cdot)$ on $\Omega$, we
first obtain the existence and uniqueness and the global
H\"older regularity of the weak solution to the Robin boundary value problem for
the second-order elliptic equation of divergence form
\begin{equation}\label{Robin-Pro}\tag{RP}
\begin{cases}
-\mathrm{div}(A\nabla u) = 0~~&\text{in}~\Omega,\\
A\nabla u\cdot \boldsymbol{\nu}+\beta u = f~~&\text{on}~\partial \Omega,
\end{cases}
\end{equation}
where the coefficient matrix $A:=\{a_{ij}\}_{i,j=1}^n$ is real-valued,
bounded, and measurable and satisfies the \emph{uniform ellipticity condition},
that is, there exists a positive constant $\mu_0\in(0,1]$ such that, for any $x\in \Omega$ and $\xi:=(\xi_1,\cdots,\xi_n)\in\mathbb{R}^n$,
\begin{equation}\label{eq1.2}
\mu_0|\xi|^2\le\sum_{i,j=1}^na_{ij}(x)\xi_i\xi_j\le\mu_0^{-1}|\xi|^2.
\end{equation}
Here, $\boldsymbol{\nu}:=(\nu_1,\ldots,\nu_n)$ denotes the \emph{outward unit normal} to $\partial\Omega$. Meanwhile,
we also establish the boundary Harnack inequality for the weak solution to the Robin problem \eqref{Robin-Pro}
and the existence and the regularity estimate of Green's functions associated
with the Robin problem \eqref{Robin-Pro}. As applications of these results,
we further prove the existence of harmonic measures on $\partial\Omega$
induced by the Robin problem \eqref{Robin-Pro} and the mutual absolute continuity between such a harmonic
measure and the surface measure $\sigma$ on $\partial\Omega$.
The main results obtained in this article extend the corresponding
results established by David et al. \cite{DDEMM24} via weakening the assumption
$\beta\equiv c$, with $c\in(0,\infty)$ being a given constant, considered in \cite{DDEMM24}
to the assumption that $\beta$ satisfies \eqref{eq1.1}. We also point out that, when considering the
existence of harmonic measures on $\partial\Omega$ associated with the Robin problem \eqref{Robin-Pro} and
the mutual absolute continuity between the harmonic measure and the surface measure $\sigma$ on $\partial\Omega$,
the additional assumption $\beta\ge a_0$ on $\partial\Omega$ is needed.

Over the past few decades, extensive research at the intersection of analysis, PDEs, and geometric measure theory has sought to
comprehensively understand the relationship between the solvability of boundary value problems on rough domains or the regularity of weak solutions to
boundary value problems on rough domains and the geometric properties of the domain under consideration (see, for example,
\cite{AHMMT20,AHMMMTV16,D00,DFM21,DFM23,DS93,dl21,HM14,K94}). In particular, the study on regularity of weak solutions to
various boundary value problems of general elliptic equations on rough domains has recently attracted considerable interest
(see, for example, \cite{BBC08,D00,DDEMM24,DFM21,DFM23,dl21,K94,yyy21}). A central motivation for this framework is to establish a
precise correspondence between the regularity properties of weak solutions and the geometric
features of the domains under consideration.

Furthermore, as a natural extension of the Dirichlet and the Neumann boundary value problems,
the Robin type boundary value problem is ubiquitous in physics, biology, and related disciplines
(see, for example, \cite{BF07,DDEFMM,DDEMM24,GFS03,GFS06,R05,w84}) and
has been studied extensively from a variety of perspectives (see, for example, \cite{AW03,BBC08,CK16,D06,DDEMM24,dl21,LS04,V12,yyz24,yyy18}).
In particular, Arendt and Warma \cite{AW03} showed that the Robin boundary condition is intermediate between the Dirichlet and the
Neumann boundary conditions and the corresponding operator semigroup satisfies a ``sandwiched" property.
Daners \cite{D00} investigated the well-posedness of the Robin boundary value problem for general elliptic equations on general domains in $\mathbb{R}^n$.
Nittka \cite{n11} and  V\'elez-Santiago \cite{V12} established global H\"older
regularity estimates for the Robin boundary value problem for general elliptic equations on bounded Lipschitz domains.
The solvability of the $L^p$ Robin problem for the Laplace operator on
bounded Lipschitz domains was addressed by Lanzani and Shen \cite{LS04} (see also \cite{yyy18}).
Meanwhile, the global Sobolev regularity for the Robin type boundary value problem on Lipschitz domains or
more general rough domains was studied by Dong and Li \cite{dl21}, as well as by Yang et al. \cite{yyy21}.
Recently, David et al. \cite{DDEMM24} established the existence of harmonic measures
associated with the Robin problem \eqref{Robin-Pro} and the mutual absolute continuity between such harmonic measure and
the surface measure on $\partial \Omega$, under the assumption that $\beta$ is a given positive constant.

To describe the main results of this article, we first recall some necessary assumptions and concepts.
	
\begin{assumption}\label{Ass}
Let $n\ge 2$, $s\in(n-2,n)$, and $\Omega\subset \mathbb{R}^n$ be an open set. The pair
$(\Omega,\sigma)$ is called a \emph{one-sided non-tangentially accessible
{\rm (for short, NTA)} pair of $s$-dimension} if it satisfies the
following geometric conditions (A1)--(A3):
\begin{enumerate}
\item[{\rm (A1)}] (\emph{Ahlfors Regularity}) The boundary of $\Omega$,
denoted by $\partial \Omega$, is $s$-\emph{Ahlfors regular}, that is, there exist a surface measure $\sigma$ supported
in $\partial \Omega$ and a positive constant $C$ such that, for any $y\in \partial \Omega$ and $r\in (0,\mathrm{diam\,}(\partial \Omega))$,
$$ C^{-1}r^s \le \sigma(\partial\Omega\cap B(y,r))\le Cr^s, $$
where $\mathrm{diam\,}(\partial \Omega):=\sup\{|x-y|:x,y\in \partial \Omega\}$ and $B(y,r):=\{z\in \mathbb{R}^n:|z-y|<r\}$.
			
\item[{\rm(A2)}] (\emph{Interior Corkscrew Condition}) There exists
a constant $C \in (0,1)$ such that, for any $y \in\partial \Omega$ and
$r\in (0,\mathrm{diam\,}(\Omega))$, there exists a point $x_y\in \Omega\cap B(y,r)$ such that $B(x_y,Cr) \subset \Omega \cap B(y,r)$.
The point $x_y$ is called a \emph{corkscrew point} with
respect to $\Omega\cap B(y,r)$.
			
\item[{\rm(A3)}] (\emph{Harnack Chains})
There exists a uniform constant $M\in[1,\infty)$ such that, for any $x,y\in\Omega$, there exists an $(M,N)$-Harnack chain
connecting $x$ and $y$, with $N\in\mathbb{N}$ depending only on $M$ and the ratio $|x-y|/\min\{\delta(x),\delta(y)\}$,
that is, there exists a chain of open balls $B_1,\ldots,B_N\subset\Omega$ such that
$x\in B_1$, $y\in B_N$, $B_k\cap B_{k+1}\neq\emptyset$ for any $k\in\{1,\ldots,N-1\}$, and $M^{-1}\mathrm{diam\,}
(B_k)\le\mathrm{dist\,}(B_k,\partial\Omega)\le M\mathrm{diam\,}(B_k)$ for any $k\in\{1,\ldots,N\}$.
Here, for any $x\in\Omega$, $\delta(x):=\mathrm{dist\,}(x,\partial\Omega):=
\inf\{|x-z|:z\in\partial\Omega\}.$
\end{enumerate}
\end{assumption}

All the constants appearing in Assumption \ref{Ass} are called the \emph{geometric constants} for $(\Omega,\sigma)$. Moreover, the domain $\Omega$ is said to have the \emph{{\rm NTA} property}
if $\Omega$ satisfies  both (A2) and (A3) of Assumption \ref{Ass}.
	
\begin{remark}
Let $\sigma$ be as in Assumption \ref{Ass}. It is easy to check that $\sigma$ is a \emph{doubling measure},
that is, there is a positive constant $C$ such that, for any $y \in \partial \Omega$ and $r\in(0,\mathrm{diam\,}(\partial\Omega))$,
\begin{equation*}
\sigma(\partial \Omega \cap B(y,2r)) \le C \sigma(\partial\Omega\cap B(y,r)).
\end{equation*}
\end{remark}
	
We point out that many known domains satisfy Assumption \ref{Ass}. Examples of such domains include Lipschitz domains in $\mathbb{R}^n$,
the complement of the four-corner Cantor set in the plane (see, for instance, \cite[Example 1]{DDEMM24}), and the
Koch snowflake (see, for instance, \cite[Chapter 1]{DS93}).
Moreover, numerous domains with fractal boundaries, including the Koch snowflake and its analogous with self-similar structures
commonly encountered in practical applications (see, for example, \cite{BBC08,S06}), also satisfy Assumption \ref{Ass}.

Let $p\in[1,\infty)$. Denote by $W^{1,p}(\Omega)$ the \emph{Sobolev space} on $\Omega$ equipped with the \emph{norm}
$$\|f\|_{W^{1,p}(\Omega)}:=\|f\|_{L^p(\Omega)}+\|\nabla f\|_{L^p(\Omega)}.$$
Here and in what follows, $\nabla f:=(\partial_1 f,\ldots,\partial_n f)$ denotes the \emph{gradient} of $f$ and $\{\partial_if\}_{i=1}^n$
are the \emph{distributional derivatives} of $f$.
	
Now, we give the definition of weak solutions for the Robin problem \eqref{Robin-Pro}.
\begin{definition}
Let $n\ge 2$, $s\in (n-2,n)$, and $(\Omega,\sigma)$ satisfy Assumption \ref{Ass}. Suppose that the function $\beta$ satisfies \eqref{eq1.1}.
Then $u\in W^{1,2}(\Omega)$ is called a \emph{weak solution} to the Robin problem \eqref{Robin-Pro} if,
for any $\varphi\in C_{\rm c}^{\infty}(\mathbb{R}^n)$ (the set of all infinitely differentiable
functions on $\mathbb{R}^n$ with compact support),
\begin{equation*}
\int_{\Omega} A\nabla u \cdot \nabla \varphi\,dx + \int_{\partial \Omega}\beta \mathrm{Tr}(u)\varphi\,d\sigma
= \int_{\partial \Omega} f\varphi\,d\sigma,
\end{equation*}
where $\mathrm{Tr}(u)$ denotes the trace of $u$ on $\partial\Omega$.
\end{definition}
Recall that the \emph{trace} of $u\in W^{1,2}(\Omega)$, denoted by $\mathrm{Tr}(u)$,
is defined as, for any $x\in \partial \Omega$,
\begin{equation}\label{eq1.4}
\mathrm{Tr} (u)(x) := \lim\limits_{\genfrac{}{}{0pt}{}{y\in \gamma(x)}
{\delta(y)\to 0}}\fint_{B(y,\delta(y)/2)} u(z)\,dz
\end{equation}
when the limit in the right-hand side of \eqref{eq1.4} exists,
where $\gamma(y):=\{z\in \Omega:|y-z| <\delta(y)\}$ with $\delta(y) := \mathrm{dist\,} (y,\partial \Omega)$. It was shown in
\cite[Theorem 6.6]{DFM23} that, for any given $u\in W^{1,2}(\Omega)$ and for almost every $x\in \partial \Omega$ with respect to $\sigma$,
the limit in the right-hand side of \eqref{eq1.4} exists and hence
$\mathrm{Tr} (u)(x)$ is well-defined.
	
Let $\alpha\in(0,1]$ and $O\subset\mathbb{R}^n$ be a subset.
We now recall the definition of the H\"older space $C^{0,\alpha}(O)$.
A function $u$ is said to belongs to the \emph{H\"older space}
$C^{0,\alpha}(O)$ if $u$ is bounded and continuous on $O$ and
$[u]_{C^{0,\alpha}(O)} <\infty$, where the H\"older semi-norm
$[u]_{C^{0,\alpha}(O)}$ is defined by setting
$$
[u]_{C^{0,\alpha}(O)}:= \sup\limits_{\genfrac{}{}{0pt}{}{x,y\in O}{x\ne y}} \dfrac{|u(x)-u(y)|}{|x-y|^{\alpha}}.$$
Meanwhile, we define the \emph{H\"older norm}
$$ \|u\|_{C^{0,\alpha}(O)}:= \sup\limits_{x\in O} |u(x)| + [u]_{C^{0,\alpha}(O)}.$$
Moreover, for any open set $O$ in $\mathbb{R}^n$, we denote by $\overline{O}$ the \emph{closure} of $O$ in $\mathbb{R}^n$.
	
Next, we give the first main result of this article.
\begin{theorem}\label{th1.1}
Let $n\ge 2$, $s\in (n-2,n)$, and $(\Omega,\sigma)$ satisfy Assumption \ref{Ass} with $\Omega$ being bounded.
Assume that $\beta$ is as in \eqref{eq1.1}.
If $f\in L^p(\partial \Omega,\sigma)$ with $p\in(s/(s+2-n),\infty]$,
then there exists a unique weak solution $u\in W^{1,2}(\Omega)$ to the Robin problem \eqref{Robin-Pro}.
		
Furthermore, there exist constants $\alpha_0\in (0,1]$ and $C\in (0,\infty)$, depending only on $n$, the geometric constants for $(\Omega,\sigma)$, and $\mu_0$ in \eqref{eq1.2}, such that, for
any given $\alpha\in (0,\alpha_0)$, $u\in C^{0,\alpha}(\overline{\Omega})$ and
\begin{equation}\label{eq1.5}
\|u\|_{C^{0,\alpha}(\overline{\Omega})}\le C\|f\|_{L^p(\partial\Omega,\sigma)}.
\end{equation}
\end{theorem}

\begin{remark}
Let $(\Omega,\sigma)$ and $\beta$ be as in Theorem \ref{th1.1}, and let $u\in W^{1,2}(\Omega)$ be the weak solution
to the Robin problem \eqref{Robin-Pro} with data $f$.
\begin{itemize}
\item[{\rm(i)}] David et al. \cite[Theorem 4.1]{DDEMM24} proved that $u\in C^{0,\alpha}(\overline{\Omega})$
for some $\alpha\in(0,1)$, under the assumptions that $\beta\equiv c$ with $c\in(0,\infty)$ being a given constant
and $f\in C^{0,\alpha}(\partial\Omega)$.

Moreover, when $\Omega$ is a bounded Lipschitz domain in $\mathbb{R}^n$ and $\beta\in L^\infty(\partial\Omega,\sigma)$,
Nittka \cite[Theorem 3.14]{n11} proved that $u\in C^{0,\alpha}(\overline{\Omega})$ for some $\alpha\in(0,1)$ and also obtained
the estimate \eqref{eq1.5}.

Thus, the global H\"older regularity for the Robin problem \eqref{Robin-Pro} established in Theorem \ref{th1.1} extends
the corresponding results obtained by David et al. \cite[Theorem 4.1]{DDEMM24} and Nittka \cite[Theorem 3.14]{n11}
via weakening, respectively, the assumptions on both the function $\beta$ and the data $f$ and the assumptions on both the domain $\Omega$ and
the function $\beta$.

\item[{\rm(ii)}] The method used in this article to prove Theorem \ref{th1.1} is different from that used by David et al. \cite[Theorem 4.1]{DDEMM24}
and Nittka \cite[Theorem 3.14]{n11}. Precisely, by first establishing a boundary Harnack inequality (see \cite[Theorem 4.4]{DDEMM24})
and then obtaining oscillation estimates (see \cite[Theorem 4.5]{DDEMM24}) for weak solutions to the local version of the Robin problem \eqref{Robin-Pro},
David et al. \cite[Theorem 4.1]{DDEMM24} proved $u\in C^{0,\alpha}(\overline{\Omega})$ for some $\alpha\in(0,1)$.
Moreover, Nittka \cite[Theorem 3.14]{n11} showed $u\in C^{0,\alpha}(\overline{\Omega})$ by applying a reflection technique
related to Lipschitz domains. However, the methods used in the proofs of \cite[Theorem 4.1]{DDEMM24} and \cite[Theorem 3.14]{n11}
are not valid in the setting of Theorem \ref{th1.1}. Indeed, since a part of the boundary of the domain $\Omega$ as in Theorem \ref{th1.1}
may not be a graph of a function, the reflection technique used in the proof of \cite[Theorem 3.14]{n11} is not applicable
in the setting of Theorem \ref{th1.1}. Moreover, although we can establish a boundary Harnack inequality (see Theorem \ref{th1.2})
similar to \cite[Theorem 4.4]{DDEMM24} in the setting of Theorem \ref{th1.1}, we cannot obtain an oscillation estimate as in
\cite[Theorem 4.5]{DDEMM24} for weak solutions to the local Robin problem, due to the generality of the function $\beta$ considered
in Theorem \ref{th1.1}.

To overcome these difficulties as above, we borrow some ideas from V\'elez-Santiago \cite{V12} and Daners \cite{D00}.
Precisely, by applying the method used in the proofs of \cite[Theorem 4.1]{D00} and \cite[Theorem 3.1]{V12}, in Lemma \ref{le3.1}
we first prove that the weak solution $u$ to the Robin problem \eqref{Robin-Pro} with data $f$ is bounded on $\overline{\Omega}$. Furthermore,
in Lemma \ref{le3.2} we obtain some oscillation estimates for weak solutions to the local Robin problem with the homogeneous boundary condition
by using a level set method used in the proof of \cite[Lemma 4.1]{V12}. Then, by using the boundedness of the weak solution $u$
obtained in Lemma \ref{le3.1} and the oscillation estimate obtained in Lemma \ref{le3.2}, we prove Theorem \ref{th1.1}.
\end{itemize}
\end{remark}

It is well known that the Harnack inequality is a fundamental and powerful tool in PDEs (see, for example, \cite{GT83,K94}).
We obtain the boundary Harnack inequality for the Robin problem \eqref{Robin-Pro} as follows.
	
\begin{theorem}\label{th1.2}
Let $n\ge 2,\ s\in (n-2,n)$, and $(\Omega,\sigma)$ satisfy Assumption \ref{Ass} with $\Omega$ being bounded.
Assume that $\beta$ satisfies \eqref{eq1.1}, $K\in(1,\infty)$ is a sufficiently large constant depending only
on the geometric constants for $(\Omega,\sigma)$, $x_0\in\partial\Omega$, and $r\in(0,\mathrm{diam\,}(\Omega)/(4K)]$.
Let $u\in W^{1,2}(B(x_0,2K^2r)\cap\Omega)$ be a non-negative weak solution to the Robin problem
\begin{equation}\label{eq1.6}
\begin{cases}
-\operatorname{div} (A\nabla u) = 0&\text{in}~~B(x_0,2K^2r),\\
A\nabla u\cdot \boldsymbol{\nu} +\beta u =0&\text{on}~~B(x_0,2K^2r)\cap\partial\Omega.
\end{cases}
\end{equation}
Then there exists constants $c_0,\eta\in(0,1)$ small enough depending only on the geometric constants for $(\Omega,\sigma)$
and $\mu_0$ in \eqref{eq1.2} such that, if
\begin{equation}\label{eq1.7}
\|\beta\|_{L^{q_0}(B(x_0,2Kr)\cap\partial\Omega)}\le c_0
\end{equation}
with $q_0\in(1,\infty]$ as in \eqref{eq1.1}, then
\begin{equation}\label{eq1.8}
\inf\limits_{x\in B(x_0,r)\cap\Omega}u(x) \ge\eta\sup\limits_{x\in B(x_0,r)\cap\Omega} u(x).
\end{equation}
\end{theorem}

\begin{remark}
When $\beta\equiv c$ with $c\in(0,\infty)$ being a given constant, the conclusion of Theorem \ref{th1.2} was obtained
in \cite[Theorem 4.4]{DDEMM24}. We also point out that, when $n\ge3$ and $\beta\equiv c$, by taking $q_0=s/(s+2-n)$ in \eqref{eq1.7},
we find that the assumption \eqref{eq1.7} on $\beta$ and the condition \cite[(4.15)]{DDEMM24} on $\beta$ in
\cite[Theorem 4.4]{DDEMM24} [namely $cr^{2-n}\sigma(B(x_0,r))\le c_0$] are consistent.

We prove Theorem \ref{th1.2} by applying some methods similar to those used in the proof of \cite[Theorem 4.4]{DDEMM24}.
Compared with the proof of \cite[Theorem 4.4]{DDEMM24}, the main new ingredients
appearing in the proof of Theorem \ref{th1.2} are the ingenious application of several Poincar\'e type inequalities
established in \cite{DFM23} and \cite{DDEMM24}.
\end{remark}
	
Moreover, Green's functions are a natural and fundamental tool to study weak solutions
to various boundary value problems and hence have attracted extensive attention in the field of PDEs
(see, for example, \cite{CK14,DDEFMM,DFM23,dk21,DK09,FL24,gw82,hs25}). In particular, Choi and Kim \cite{CK14}
have studied Green's functions for second-order elliptic and parabolic systems in divergence form with Robin-type boundary conditions
on rough domains. Inspired by the work of Choi and Kim \cite{CK14}, to study the harmonic measure related to the Robin problem
\eqref{Robin-Pro} via using the theory of Green's functions we study the existence and regularity of Green's function associated  with
the Robin problem \eqref{Robin-Pro}. Let us state this result as follows.
	
\begin{theorem}\label{th1.3}
Let $n\ge 2$, $s\in (n-2,n)$, and $(\Omega, \sigma)$ satisfy Assumption \ref{Ass} with $\Omega$ being bounded.
Assume that the function $\beta$ satisfies \eqref{eq1.1}. Then, for any $y\in \Omega$, there exists
a non-negative function $G_R(\cdot,y)$ on $\Omega$ such that, for any $r\in(0,\mathrm{diam\,}(\Omega))$,
\begin{equation*}
G_R(\cdot,y)\in W^{1,2}(\Omega\setminus B(y,r))\cap W^{1,1}(\Omega)
\end{equation*}
and, for any $\varphi\in C^{\infty}_{\rm c}(\mathbb{R}^n)$ and $y\in\Omega$,
\begin{equation}\label{eq1.9}
\int_{\Omega} A\nabla G_R(x,y) \cdot \nabla \varphi(x)\,dx + \int_{\partial \Omega}\beta(x)G_R(x,y)\varphi(x)\,d\sigma(x) = \varphi(y),
\end{equation}
where the function $G_R$ on $\Omega\times\Omega$ is called {\rm Green's function} for the Robin problem \eqref{Robin-Pro}. Moreover, if $u$ is the
weak solution for the Robin problem \eqref{Robin-Pro} with data $f\in L^2(\partial\Omega,\sigma)$, then, for any $x\in \Omega$,
\begin{equation}\label{eq1.10}
u(x)=\int_{\partial \Omega}G_R(y,x)f(y)\,d\sigma(y).
\end{equation}
Furthermore, Green's function $G_R$ has the following two properties:
\begin{itemize}	
\item[{\rm(i)}] For any $x,y\in \Omega$ with $x\ne y$,
\begin{equation}\label{eq1.11}
0 \le G_R(x,y) \le \begin{cases}
C\left[ 1 + \log\left(\dfrac{1}{|x-y|} \right) \right]&\text{if }n=2,\\
\dfrac{C}{|x-y|^{n-2}}&\text{if }n\ge 3,
\end{cases}
\end{equation}
where the positive constant $C$ depends only on $n$, the geometric
constants for $(\Omega,\sigma)$, and $\mu_0$ in \eqref{eq1.2}.
		
\item[{\rm(ii)}] Let $n\ge 3$. Then there exists $\delta\in (0,1]$
depending only on $n$, the geometric constants for $(\Omega,\sigma)$,
and $\mu_0$ in \eqref{eq1.2} such that, for any $x_1,x_2,y\in \Omega$ satisfying $|x_1-x_2|<|x_1-y|/2$,
\begin{equation}\label{eq1.12}
|G_R(x_1,y) - G_R(x_2,y)|\le C\dfrac{|x_1-x_2|^\delta}{|x_1-y|^{n-2+\delta}},
\end{equation}
where the positive constant $C$ depends only on $n$, the geometric
constants for $(\Omega,\sigma)$, and $\mu_0$ in \eqref{eq1.2}.
\end{itemize}
\end{theorem}

We prove the existence of Green's function $G_R$ by following the standard approach
developed by Gr\"uter and Widman \cite{gw82} and obtain the representation formula
\eqref{eq1.10} by applying the method used in the proof of \cite[(5.4)]{DDEMM24}. Moreover,
we obtain the upper bound estimate \eqref{eq1.11} by using Gaussian upper bound estimates for
heat kernels associated with the Robin problem \eqref{Robin-Pro} essentially established by Choi and Kim \cite{CK14}.
Furthermore, we prove the H\"older regularity \eqref{eq1.12} of $G_R$ by applying
\eqref{eq1.11} and the interior H\"older estimate for weak solutions to the local Robin problem
(see, for instance, \cite[Lemma 11.30]{DFM23}). We also point out that, when $\beta\equiv c$ with $c\in(0,\infty)$
being a given constant, the existence of $G_R$ and \eqref{eq1.10} were obtained in \cite[Theorem 5.6]{DDEMM24}.
	
\begin{remark}
By \eqref{eq1.9}, it is easy to find that $G_R^*(x,y) = G_R(y,x)$ for any $x,y\in\Omega$, where $G_R^*$ denotes Green's function
for the Robin problem
$$ \begin{cases}
-\mathrm{div} (A^T\nabla u) = 0&\quad \text{in}~\Omega,\\
A^T\nabla u\cdot \boldsymbol{\nu} + \beta u = f&\quad \text{on}~\partial \Omega,
\end{cases} $$
where $A^T$ denotes the transport matrix of $A$. Therefore, the same conclusion as in Theorem \ref{th1.3} also holds for $G_R^*$.
By this, we further conclude the H\"older continuity of $G_R(\cdot,\cdot)$ with respect to the second variable.
\end{remark}
	
As applications of the global H\"older regularity in Theorem \ref{th1.1}, the boundary Harnack inequality in
Theorem \ref{th1.2}, and the representation formula \eqref{eq1.10} in Theorem \ref{th1.3},
we further prove the existence of harmonic measures induced by the
Robin problem \eqref{Robin-Pro} and the mutual absolute continuity between the harmonic measure and
the surface measure $\sigma$ (see Theorems \ref{th1.4} and \ref{th1.5}).

The study of the relationship between geometric properties of a domain $\Omega$ and the absolute continuity of the
harmonic measure associated with the Dirichlet problem on $\Omega$ with respect to
the surface measure on the boundary $\partial\Omega$ has a long history. In particular, F. Riesz and M. Riesz \cite{RR20} showed
that, for a simply connected domain $\Omega$ in the complex plane with a rectifiable
boundary $\partial\Omega$, the harmonic measure is absolutely continuous
with respect to the arclength measure on $\partial\Omega$. Dahlberg \cite{D77} proved the mutual absolute continuity between the harmonic measure
and the surface measure for any Lipschitz domains in $\mathbb{R}^n$ with $n\ge2$. David and Jerison \cite{dj90} and Semmes
\cite{s89} independently extended the result of Dahlberg \cite{D77} to the case of chord-arc domains in $\mathbb{R}^n$ with $n\ge2$.
Recently, Azzam et al. \cite{AHMMT20} obtained a complete geometric characterization for the mutual absolute continuity between
the harmonic measure and the surface measure on the boundary of a rough domain $\Omega$ in $\mathbb{R}^n$. We also
refer to \cite{AGMT23,AHMMMTV16,DM23,HM14,HMM16} for more recent progress on the absolute continuity of the
harmonic measure for the Dirichlet problem with respect to the surface measure.
	
Furthermore, David et al. \cite{DDEMM24} creatively developed the theory of the harmonic measure associated with
the Robin problem \eqref{Robin-Pro} with $\beta\equiv c$ on the domain $\Omega$ satisfying Assumption \ref{Ass}
and established the qualitative mutual absolute continuity between such harmonic measure and the surface measure on $\partial\Omega$.

To describe our results for the harmonic measure associated with the Robin problem \eqref{Robin-Pro}, we first
recall the concept of absolutely continuity between two measures.
Given two Borel measures $\mu_1,\mu_2$ defined on a set $\Sigma$, the measure $\mu_1$
is said to be \emph{absolutely continuous} with respect to $\mu_2$,
denoted by $\mu_1 \ll \mu_2$, if $\mu_2(E) = 0$ implies $\mu_1(E) = 0$ for any Borel measurable subset $E\subset\Sigma$.
	
\begin{theorem}\label{th1.4}
Let $n\ge 2$, $s\in (n-2,n)$, and $(\Omega, \sigma)$ satisfy Assumption \ref{Ass} with $\Omega$ being bounded.
Assume that $\beta$ satisfies \eqref{eq1.1} and $\beta \ge a_0$ on $\partial \Omega$ with $a_0$ being a given positive constant.
For any Robin boundary value $f\in C(\partial\Omega)$, let $u_f$ be the unique weak solution to the Robin problem \eqref{Robin-Pro}.
Then there exists a family of Radon measures
$\{w_{R}^x \}_{x\in \Omega}$, supported in $\partial \Omega$, such that, for any $x\in\Omega$,
\begin{equation}\label{eq1.13}
u_f(x) = \int_{\partial \Omega} f(z)~dw_{R}^x(z).
\end{equation}
The measure $w^x_{R}$ is called the {\rm Robin harmonic measure} associated with the Robin problem \eqref{Robin-Pro}.
		
Furthermore, for any $x\in\Omega$,
\begin{equation}\label{eq1.14}
\sigma \ll w^x_{R} \ll \sigma
\end{equation}
and
\begin{equation}\label{eq1.15}
\dfrac{dw_R^x}{d\sigma} = G_R(\cdot,x).
\end{equation}
\end{theorem}

We prove Theorem \ref{th1.4} by applying the Riesz representation theorem, the maximum principle,
the H\"older regularity in Theorem \ref{th1.1}, the boundary Harnack inequality in Theorem \ref{th1.2},
and the representation formula \eqref{eq1.10}. It is worth pointing out that, when $\beta\equiv c$ with $c\in(0,\infty)$
being a given constant, Theorem \ref{th1.4} is precisely \cite[Theorem 1.2]{DDEMM24}.
Therefore, Theorem \ref{th1.4} extends \cite[Theorem 1.2]{DDEMM24} by weakening the assumption on $\beta$.

\begin{remark}
Let $(\Omega, \sigma)$, $\beta$ be the same as in Theorem \ref{th1.4}, and $x\in\Omega$.
By \eqref{eq1.15}, we find that $G(\cdot,x)$ is the Radon--Nikodym derivative of the
harmonic measure $w_{R}^x$ with respect to the surface measure $\sigma$
[also called the Poisson kernel for the Robin problem \eqref{Robin-Pro}].
\end{remark}
	
\begin{theorem}\label{th1.5}
Let $n\ge 2$, $s\in (n-2,n)$, and $(\Omega,\sigma)$ satisfy Assumption \ref{Ass} with $\Omega$ being bounded. Assume that $\beta$
satisfies \eqref{eq1.1} and $\beta \ge a_0$ on $\partial \Omega$ with $a_0$ being a given positive constant and $\{w_R^x\}_{x\in\Omega}$
are the corresponding harmonic measures as in Theorem \ref{th1.4}.
Then there exists a constant $C\in[1,\infty)$ depending on the geometric constants for $(\Omega,\sigma)$
and $\mu_0$ in \eqref{eq1.1} such that, for any $x_0\in \partial\Omega$ and $r\in(0,\mathrm{diam\,}(\Omega)/(4K)]$
satisfying $\|\beta\|_{L^{q_0}(B(x_0,2Kr)\cap\partial\Omega)}\le c_0$, where $c_0\in(0,1)$ and $K\in(1,\infty)$ are as in
Theorem \ref{th1.2}, for any $x\in\Omega\setminus B(x_0,Cr)$, and for any $E\subset B(x_0,r)\cap\partial\Omega$,
\begin{equation}\label{eq1.16}
C^{-1}\dfrac{\sigma(E)}{\sigma(B(x_0,r)\cap\partial\Omega)}\le\dfrac{w_R^x(E)}{w_R^x(B(x_0,r)\cap\partial\Omega)}
\le C\dfrac{\sigma(E)}{\sigma(B(x_0,r)\cap\partial\Omega)}.
\end{equation}
\end{theorem}

Theorem \ref{th1.5} provides a quantitative characterization of the mutual absolute continuity between harmonic measures $\{w_R^x\}_{x\in\Omega}$
and the surface measure $\sigma$ at small scales. Following the approach used in the proof of \cite[Theorem 1.3]{DDEMM24},
we show Theorem \ref{th1.5} by applying the boundary Harnack
inequality obtained in Theorem \ref{th1.2} for Green's function $G_R$ and using \eqref{eq1.15}. We also
point out that, when $\beta\equiv c$ with $c\in(0,\infty)$ being a given constant, the conclusion of Theorem \ref{th1.5}
was obtained in \cite[Theorem 1.3]{DDEMM24}.
	
The remainder of this article is organized as follows.
	
In Section \ref{S2}, we first prove a trace embedding theorem for the Sobolev space $W^{1,2}(\Omega)$ and then prove the existence and
uniqueness of the weak solution to the Robin problem \eqref{Robin-Pro}. Furthermore, we also recall several Poincar\'e type inequality
essentially obtained in \cite{DDEMM24,DFM23}.
	
In Section \ref{S3}, we prove Theorem \ref{th1.1} by first establishing the $L^{\infty}$-priori estimate
for the weak solution to the Robin problem \eqref{Robin-Pro} and oscillation estimates for weak solutions
to the local Robin problem with homogeneous boundary conditions.

In Section \ref{S5}, we show Theorem \ref{th1.2} by first obtaining the Moser type estimate for weak solutions to the local Robin problem.
Moreover, in Section \ref{S4}, we give the proof of Theorem \ref{th1.3}.
	
In Section \ref{S6}, we prove Theorems \ref{th1.4} and \ref{th1.5} by using
the Riesz representation theorem, the maximum principle for the Robin problem \eqref{Robin-Pro},
the boundary Harnack inequality obtained in Theorem \ref{th1.2}, and properties of Green's function $G_R$ given in Theorem \ref{th1.3}.
	
We end this introduction by making some conventions on symbols. Throughout this article, we always denote by $C$ a \emph{positive constant} which is
independent of the main parameters involved, but it may vary from line to line. The \emph{symbol} $f\lesssim g$ means that $f\le Cg$.
If $f\lesssim g$ and $g\lesssim f$, then we write $f\sim g$.
If $f\le Cg$ and $g=h$ or $g\le h$, we then write $f\lesssim g=h$ or $f\lesssim g\le h$. For any ball $B:=B(x_B,r_B)$ in $\mathbb{R}^n$,
with $x_B\in\mathbb{R}^n$ and $r_B\in (0,\infty)$, and for any $k \in(0,\infty)$, let $k B:=B(x_B,k r_B)$.
For any subset $E$ in $\mathbb{R}^n$, we denote by $\mathbf{1}_{E}$ its \emph{characteristic function}.
We also let $\mathbb{N}:=\{1,2,\ldots\}$ and $\mathbb{Z}_+:=\mathbb{N}\cup\{0\}$. For any given normed spaces $\mathcal X$
and $\mathcal Y$ with the corresponding norms $\|\cdot\|_{\mathcal X}$ and $\|\cdot\|_{\mathcal Y}$, the
\emph{symbol} ${\mathcal X}\hookrightarrow{\mathcal Y}$ means that, for any $f\in \mathcal X$, then $f\in
\mathcal Y$ and $\|f\|_{\mathcal Y}\lesssim \|f\|_{\mathcal X}$ with the implicit positive constant independent of $f$.
Moreover, for any $q\in[1,\infty]$, we denote by $q'$ its \emph{conjugate exponent}, that is, $1/q+1/q'= 1$.
For a metric space $\Sigma$ equipped with a non-atomic measure $\mu$ on $\Sigma$,
a Borel set $E\subset\Sigma$ with $\mu(E)\in(0,\infty)$, and a locally integrable function $f$ on $\Sigma$,
let
\begin{equation*}
\fint_{E}f\,d\mu:=\frac{1}{\mu(E)}\int_{E}f\,d\mu.
\end{equation*}
Finally, in all proofs we consistently retain the symbols
introduced in the original theorem (or related statement).
	
\section{Preliminaries\label{S2}}
	
In this section, we first present a trace embedding theorem for the Sobolev
space $W^{1,2}(\Omega)$, which enables us to prove the existence and uniqueness of the weak solution to the Robin problem \eqref{Robin-Pro}.
Furthermore, we recall several Poincar\'e type inequalities essentially obtained in \cite{DDEMM24,DFM23}, which serve as a fundamental tool
for establishing the H\"older regularity and the boundary Harnack inequality for the weak solution to
the Robin problem \eqref{Robin-Pro}.
	
The following Sobolev embedding theorem was established in \cite[Theorem 2.1]{DDEMM24}.
	
\begin{lemma}\label{le2.1}
Let $n\ge 2$, $s\in (n-2,n)$, and $(\Omega,\sigma)$ satisfy Assumption \ref{Ass} with $\Omega$ being bounded. Then
$W^{1,2}(\Omega) \hookrightarrow L^p(\Omega)$ for any $p\in [2,2n/(n-2)]$ when
$n\ge 3$ and for any $p\in [2,\infty)$ when $n=2$. That is, for any $u\in W^{1,2}(\Omega)$,
$u\in L^p(\Omega)$ with $p\in [2,2n/(n-2)]$ when
$n\ge 3$ and with $p\in [2,\infty)$ when $n=2$ and, moreover,
$\|u\|_{L^p(\Omega)} \le C \|u\|_{W^{1,2}(\Omega)},$
where the positive constant $C$ depends only on $n$, $p$, and the geometric constants for $(\Omega,\sigma)$.
\end{lemma}
	
Recall that the \emph{fractional Sobolev space} $H(\partial \Omega,\sigma)$
on $\partial\Omega$ is defined to be the set of all $u\in L^2(\partial\Omega,\sigma)$ satisfying
$$
\|u\|^2_{\dot{H}(\partial \Omega,\sigma)} := \int_{\partial \Omega}\int_{\partial
\Omega} \dfrac{|u(x)-u(y)|^2}{|x-y|^{2(s+1)-n}}\,d\sigma(x)d\sigma(y)<\infty
$$	
(see, for example, \cite{DDEMM24,DFM21,DFM23}).
Meanwhile, for any $u\in H(\partial \Omega,\sigma)$, the \emph{Sobolev norm} $\|u\|_{H(\partial \Omega,\sigma)}$
is defined as
$$
\|u\|_{H(\partial \Omega,\sigma)}:=\|u\|_{L^2(\partial\Omega,\sigma)}
+\|u\|_{\dot{H}(\partial \Omega,\sigma)}.
$$
David et al. \cite[Theorem 6.6]{DFM23} and \cite[Theorem 2.1]{DDEMM24}
demonstrated that the trace operator $\mathrm{Tr}$ defined by \eqref{eq1.4}
is bounded from $W^{1,2}(\Omega)$ to $H(\partial \Omega,\sigma)$.
	
\begin{lemma}\label{le2.2}
Let $n\ge 2$, $s\in (n-2,n)$, and $(\Omega,\sigma)$ satisfy Assumption \ref{Ass} with $\Omega$ being bounded.
Then, for any $u\in W^{1,2}(\Omega)$, $\mathrm{Tr}(u)\in H(\partial \Omega,\sigma)$
and $\|\mathrm{Tr}(u)\|_{H(\partial \Omega,\sigma)}\le C\|u\|_{W^{1,2}(\Omega)}$ with
$C$ being a positive constant independent of $u$.
\end{lemma}
	
Based on Lemma \ref{le2.2}, to obtain a general trace embedding theorem for $W^{1,2}(\Omega)$,
we proceed in two steps: first embedding the fractional Sobolev space $H(\partial \Omega,\sigma)$
into an integer-order Sobolev space defined on a suitable metric space (see \cite[Theorem 4]{HM97}),
and then embedding this integer-order Sobolev space into the Lebesgue space
$L^p(\partial \Omega,\sigma)$ by means of a Sobolev embedding theorem on $\partial\Omega$ (see \cite[Theorem 6]{H96}).
	
Now, we recall the definition of Sobolev spaces on a suitable metric space (see, for instance, \cite[Section 3]{H96}).
\begin{definition}\label{de2.1}
Let $n\ge 2$, $s\in (n-2,n)$, and $(\Omega,\sigma)$ satisfy Assumption \ref{Ass} with $\Omega$ being bounded. For any $x,y\in\partial\Omega$,
define $d_\lambda(x,y):=|x-y|^{1-\lambda}$ for some $\lambda\in [0,1)$. Then $(\partial \Omega,d_\lambda,\sigma)$ forms a metric space.
		
The \emph{Sobolev space} $W^{1,2}(\partial \Omega,d_\lambda,\sigma)$
 is defined to be the space of all $u\in L^2(\partial \Omega,\sigma)$
satisfying that there exists a non-negative function $g\in L^2(\partial \Omega, \sigma)$ such that,
for almost every $x,y\in \partial \Omega$ with respect to $\sigma$,
\begin{equation}\label{eq2.1}
|u(x)-u(y)| \le d_\lambda(x,y)\left[g(x)+g(y)\right].
\end{equation}
Moreover, for any $u\in W^{1,2}(\partial \Omega ,d_\lambda,\sigma)$,
its \emph{norm} $\|u\|_{W^{1,2}
(\partial \Omega,d_\lambda,\sigma)}$ is defined by setting
$$\|u\|_{W^{1,2}(\partial \Omega,d_\lambda,\sigma)} := \|u\|_{L^2(\partial \Omega,\sigma)}
+\inf_{g} \|g\|_{L^2(\partial \Omega,\sigma)},$$
where the infimum is taken over all non-negative functions $g\in L^2(\partial \Omega,
\sigma)$ satisfying \eqref{eq2.1}.
\end{definition}
	
Then we have the following Sobolev trace embedding theorem.
	
\begin{theorem}\label{th2.1}
Let $n\ge2$, $s\in(n-2,n)$, and $(\Omega,\sigma)$ satisfy Assumption \ref{Ass}
with $\Omega$ being bounded. Assume that $\lambda:=(n-s)/2$.
Then the following assertions hold.
\begin{itemize}
\item[{\rm(i)}] $H(\partial\Omega,\sigma)\hookrightarrow
W^{1,2}(\partial\Omega,d_\lambda,\sigma)$ and, for any $u\in H(\partial \Omega,\sigma)$,
\begin{equation}\label{eq2.2}
\|u\|_{W^{1,2}(\partial \Omega,d_\lambda,\sigma)} \le C \|u\|_{H(\partial\Omega,\sigma)},
\end{equation}
where the positive constant $C$ depends only on $n$ and the geometric constants for $(\Omega,\sigma)$.
\item[{\rm(ii)}] When $n\ge3$, $ W^{1,2}(\partial\Omega,d_\lambda,\sigma)
\hookrightarrow L^p(\partial\Omega,\sigma)$ for any given $p\in[1,\frac{2s}{n-2}]$ and, when $n=2$, $W^{1,2}(\partial\Omega,d_\lambda,\sigma)\hookrightarrow L^p(\partial\Omega,\sigma)$
for any given $p\in[1,\infty)$.
\item[{\rm(iii)}] For any $u\in W^{1,2}(\Omega)$, $\mathrm{Tr}(u)\in L^p(\partial\Omega,\sigma)$
with $p\in[1,\frac{2s}{n-2}]$ when $n\ge3$ and $\mathrm{Tr}(u)\in L^p(\partial\Omega,
\sigma)$ with $p\in[1,\infty)$ when $n=2$. Moreover, for any $u\in W^{1,2}(\Omega)$,
$\|\mathrm{Tr}(u)\|_{L^p(\partial\Omega,\sigma)} \le C \|u\|_{W^{1,2}(\Omega)},
$
where the positive constant $C$ is independent of $u$.
\end{itemize}
\end{theorem}
\begin{proof}
(i) Let $u\in H(\partial \Omega,\sigma)$. To show $u\in W^{1,2}(\partial \Omega,
d_\lambda,\sigma)$, it suffices to find a non-negative function $g\in L^2(\partial
\Omega,\sigma)$ such that \eqref{eq2.1} holds.
		
Let $x,y\in \partial \Omega$ with $x\neq y$ and select a ball $B(x_0,r_0)$ with
some $x_0 \in \partial \Omega$ such that $x,y\in B(x_0,r_0)$ and $|x-y|\le r_0 \le 2|x-y|$.
Then, by the triangle inequality, we find that
\begin{align}\label{eq2.3}
|u(x) - u(y)|&\le |u(x) - u_{B(x_0,r_0)}| + |u(y) - u_{B(x_0,r_0)}| \notag\\
&\le 2d_\lambda(x,y)\left[G^\lambda_{|x-y|}u(x) + G^{\lambda}_{|x-y|} u(y)\right],
\end{align}
where
$$u_{B(x_0,r_0)}:= \fint_{B(x_0,r_0)\cap \partial \Omega} u\,d\sigma$$
and, for any $z\in B(x_0,r_0)\cap\partial\Omega$ and $r\in(0,\infty)$,
$$G^{\lambda}_ru(z) := \sup\limits_{t\in(0,2r)} \frac{|u(z)-u_{B(x_0,t)}|}{t^{1-\lambda}}.
$$
Let $h:=\mathrm{diam\,}(\partial \Omega)$. From \eqref{eq2.3}, we further infer that
\begin{align}\label{eq2.4}
|u(x) - u(y)| \le 2 d_\lambda(x,y) \left[G_h^{\lambda} u(x) + G^\lambda_h u(y) \right].
\end{align}
		
Next, we show $G_h^{\lambda} u \in L^2(\partial \Omega, \sigma)$. By the definition
of $G_h^{\lambda} u$, we conclude that, for any given $x\in \partial \Omega$, there
exists a ball $B(x_0,r_0)$ with $x_0\in\partial\Omega$ and $r_0\in(0,h]$ such that $x\in B(x_0,r_0)$
and
$$
G^\lambda_h u(x) \le 2 \frac{|u(x)-u_{B(x_0,r_0)}|}{r_0^{1-\lambda}},
$$
which, combined with $\lambda:=(n-s)/2$, implies that, for any $x\in\partial\Omega$,
\begin{align}\label{eq2.5}
G^\lambda_h u(x)&\le \frac{2}{r_0^{1-\lambda}} \fint_{\partial \Omega\cap B(x_0,r_0)}
|u(x) - u(z)|\,d\sigma(z)\notag\\
&\lesssim\left[\fint_{\partial\Omega \cap B(x_0,r_0)} \dfrac{|u(x)-u(z)|^2}
{|x-z|^{2(1-\lambda)}}\,d\sigma(z)\right]^{\frac12}
\lesssim\left[\int_{\partial \Omega} \dfrac{|u(x)-u(z)|^2}{|x-z|^{2(s+1)-n}}
\,d\sigma(z)\right]^{\frac12}.
\end{align}
Meanwhile, from $u\in H(\partial \Omega,\sigma)$, it follows that
$$\|u\|^2_{H(\partial \Omega,\sigma)}=\int_{\partial \Omega} |u(x)|^2\,d\sigma(x)
+\int_{\partial \Omega}\int_{\partial \Omega} \dfrac{|u(x)-u(z)|^2}
{|x-z|^{2(s+1)-n}}\,d\sigma(x)d\sigma(z)<\infty,$$
which, together with \eqref{eq2.5}, further implies that $G_h^{\lambda} u \in
L^2(\partial \Omega,\sigma )$ and
\begin{align}\label{eq2.6}
\left\|G_h^{\lambda}u\right\|_{L^2(\partial \Omega,\sigma)}\lesssim
\|u\|_{H(\partial \Omega,\sigma)}.
\end{align}
		
Let $g:=2G_h^\lambda u$. Then, by \eqref{eq2.4} and \eqref{eq2.6},
we find that $g\in L^2(\partial \Omega,\sigma)$ and \eqref{eq2.1} holds.
Therefore, $u\in W^{1,2}(\partial \Omega,d_\lambda,\sigma)$ and \eqref{eq2.2} holds.
This finishes the proof of (i).
		
(ii) From the Sobolev embedding theorem for the Sobolev space $W^{1,2}(\partial \Omega,
d_\lambda,\sigma)$ (see \cite[Theorem 6]{H96}), we infer that, when $n \ge 3$,
$ W^{1,2} (\partial \Omega,d_\lambda,\sigma) \hookrightarrow L^{p_0}(\partial \Omega,\sigma)$ with $p_0:= \frac{2s}{n-2}$. Combined with the assumption that $\partial\Omega$ is bounded, this yields $W^{1,2} (\partial \Omega,d_\lambda,\sigma)
\hookrightarrow L^{p}(\partial \Omega,\sigma)$ for any $p\in[1,p_0]$ when $n\ge 3$.
Moreover, if $n=2$, by \cite[Theorem 6(2)]{H96}, we conclude that, for any
$u\in W^{1,2}(\partial \Omega,d_\lambda,\sigma)$,
$$\int_{\partial \Omega} \exp\left\{c\dfrac{[\sigma(\partial \Omega)]^{1/2}|u(x)-u_{\partial \Omega}|}
{\mathrm{diam\,}(\partial \Omega)
\|u\|_{W^{1,2} (\partial \Omega,d_\lambda,\sigma)}}\right\}\,d\sigma(x)\lesssim1,$$
which, together with the inequality that $\kappa^p \le C_{(p)} e^{\kappa}$ for any
$\kappa\in(0,\infty)$ and $p \in [1,\infty)$, implies that, for any given
$p\in[1,\infty)$, $|u-u_{\partial \Omega}|\in L^p(\partial\Omega,\sigma)$ and
$
\|u-u_{\partial \Omega}\|_{L^p(\partial\Omega,\sigma)}\lesssim\|u\|_{W^{1,2}
(\partial \Omega,d_\lambda,\sigma)}.
$
Combined with the assumption that $\partial\Omega$ is bounded, this yields,
for any given $p\in[1,\infty)$, $u\in L^p(\partial\Omega,\sigma)$ and
$\|u\|_{L^p(\partial\Omega,\sigma)}\lesssim\|u\|_{W^{1,2} (\partial \Omega,d_\lambda,\sigma)}$,
which completes the proof of (ii).
		
(iii) From the above just proved (i) and (ii) and Lemma \ref{le2.2},
it follows that assertion (iii) also holds. This finishes the proof of
Theorem \ref{th2.1}.
\end{proof}

Next, we recall several Poincar\'e type inequalities for
the Sobolev space $W^{1,p}(\Omega)$.

\begin{lemma}\label{le2.3}
Let $n\ge2$, $s\in(n-2,n)$, and $(\Omega,\sigma)$ satisfy Assumption \ref{Ass} with $\Omega$
being bounded. Assume that $q\in[p,np/(n-p)]$ when
either $n\ge3$ and $p\in[1,2]$ or $n=2$ and $p\in[1,2)$, and assume
$q\in[2,\infty)$ when $n=2$ and $p=2$. Then there exists a constant $K\in[1,\infty)$ depending only
on the geometric constants for $(\Omega,\sigma)$ such that, for any $x\in\partial\Omega$,
$r\in(0,\mathrm{diam\,}(\Omega)]$, and $u\in W^{1,p}(\Omega)$,
\begin{equation*}
\left[\fint_{B(x,r)\cap\Omega}|u-\overline{u}|^q\,dy\right]^{\frac1q}\le C
r\left[\fint_{B(x,Kr)\cap\Omega}|\nabla u|^p\,dy\right]^{\frac1p}
\end{equation*}
with $\overline{u}:=\fint_{E}u\,dy$ for any given set $E\subset B(x,2r)\cap\Omega$
satisfying $|E|\ge c|B(x,r)\cap\Omega|$ for some constant $c\in(0,1]$, where $C$
is a positive constant depending only on $c,\ p$, and the geometric constants for
$(\Omega,\sigma)$.
\end{lemma}

The Poincar\'e type inequality stated in Lemma \ref{le2.3}
was essentially established in \cite[Theorem 5.24 and Remark 5.32]{DFM23}
(see also \cite[Lemma 2.2 and Remark 2]{DDEMM24}).

\begin{lemma}\label{le2.4}
Let $n\ge2$, $s\in(n-2,n)$, and $(\Omega,\sigma)$ satisfy Assumption \ref{Ass} with $\Omega$ being bounded.
Then there exists a constant $K\in[1,\infty)$ depending only on the geometric
constants for $(\Omega,\sigma)$ such that, for any $x\in\partial\Omega$,
$r\in(0,\mathrm{diam\,}(\Omega)/K]$, and $u\in W^{1,2}(\Omega)$,
\begin{equation*}
\left[\fint_{B(x,r)\cap\partial\Omega}|\mathrm{Tr}(u)-\overline{u}|^2
\,d\sigma(y)\right]^{\frac12}\le C
r\left[\fint_{B(x,Kr)\cap\Omega}|\nabla u|^2\,dy\right]^{\frac12}
\end{equation*}
with $\overline{u}:=\fint_{E}u\,dy$ for any given set $E\subset B(x,2r)\cap\Omega$
satisfying $|E|\ge c|B(x,r)\cap\Omega|$ for some constant $c\in(0,1]$, where $C$
is a positive constant depending only on $c,\ p$, and the geometric constants for
$(\Omega,\sigma)$.
\end{lemma}

Lemma \ref{le2.4} is precisely \cite[Lemma 2.3]{DDEMM24}.

\begin{lemma}\label{le2.5}
Let $n\ge2$, $s\in(n-2,n)$, and $(\Omega,\sigma)$ satisfy Assumption \ref{Ass} with $\Omega$ being bounded.
Let $p:=2n/(n-2)$ when $n\ge3$ and $p\in[2,\infty)$
when $n=2$, $x_0\in\partial\Omega$, and $r\in(0,\mathrm{diam\,}(\Omega)/4]$.
Then, for any $u\in W^{1,2}(\Omega)$ with $\mathrm{Tr}(u)=0$ on
$\partial B(x_0,r)\cap\Omega$,
\begin{equation*}
\left[\fint_{B(x_0,r)\cap\Omega}|u|^p\,dy\right]^{\frac1p}
\le C
r\left[\fint_{B(x_0,r)\cap\Omega}|\nabla u|^2\,dy\right]^{\frac12},
\end{equation*}
where $C$ is a positive constant depending only on $p$ and the geometric constants for $(\Omega,\sigma)$.
\end{lemma}

The Poincar\'e type inequality given in Lemma \ref{le2.5}
was essentially obtained in \cite[Corollary 7.9]{DFM23}
(see also \cite[Lemma 2.7]{DDEMM24}).

\begin{lemma}\label{le2.6}
Let $n\ge2$, $s\in(n-2,n)$, and $(\Omega,\sigma)$ satisfy Assumption \ref{Ass} with $\Omega$ being bounded.
Then there exists a constant $K\in[1,\infty)$ depending only on the geometric
constants for $(\Omega,\sigma)$ such that, for any $x\in\partial\Omega$,
$r\in(0,\mathrm{diam\,}(\Omega)]$, $E\subset B(x,r)\cap\partial\Omega$
with $\sigma(E)>0$, and $u\in W^{1,2}(\Omega)$,
\begin{equation*}
\fint_{B(x,r)\cap\Omega}|u|^2\,dy
\le C\frac{\sigma(B(x,r)\cap\partial\Omega)}{\sigma(E)}
r^2\fint_{B(x,Kr)\cap\Omega}|\nabla u|^2\,dy+2\fint_E|u|^2\,d\sigma,
\end{equation*}
where $C$ is a positive constant depending only on the geometric constants
for $(\Omega,\sigma)$.
\end{lemma}

Lemma \ref{le2.6} is exactly \cite[Lemma 2.8]{DDEMM24}.
	
To show the existence and uniqueness of the weak solution to the Robin problem \eqref{Robin-Pro}, we define the bilinear form $B[\cdot, \cdot]$
as follows: For any $u,v\in W^{1,2}(\Omega)$, let
$$
B[u,v]:= \int_{\Omega}A\nabla u\cdot \nabla v \,dx + \int_{\partial \Omega}\beta\mathrm{Tr}(u)\mathrm{Tr}(v)\,d\sigma.
$$
	
\begin{theorem}\label{th2.2}	
Let $n\ge2$, $s\in(n-2,n)$, and $(\Omega,\sigma)$ satisfy Assumption \ref{Ass} with $\Omega$ being bounded.
Assume that the non-negative function $\beta$ satisfies \eqref{eq1.1}.
Then, for any $f \in L^{p_0}(\partial \Omega, \sigma)$ with $p_0:=2s/(2s+2-n)$ when $n\ge3$ and $p_0\in(1,\infty)$ when $n=2$,
there exists a unique $u\in W^{1,2}(\Omega)$ such that, for any $\varphi\in C_{\rm c}^\infty(\mathbb{R}^n)$,
\begin{equation}\label{eq2.7}
\int_{\Omega}A\nabla u \cdot \nabla \varphi\,dx +\int_{\partial \Omega}\beta\mathrm{Tr}(u)\varphi\,d\sigma
=\int_{\partial \Omega}f\varphi\,d\sigma;
\end{equation}
that is, there exists a unique weak solution to the Robin problem \eqref{Robin-Pro}.
Furthermore,
\begin{equation}\label{eq2.8}
\|u\|_{W^{1,2}(\Omega)} \le C \|f\|_{L^{p_0}(\partial \Omega,\sigma)},
\end{equation}
where the positive constant $C$ depends only on $n$, the geometric constants for $(\Omega,\sigma)$, $\mu_0$ in \eqref{eq1.1}, and
$\|\beta\|_{L^{q_0}(\partial\Omega,\sigma)}$.
\end{theorem}
	
\begin{proof}
By \eqref{eq1.2} and \eqref{eq1.1}, we find that, for any $u\in W^{1,2}(\Omega)$,
\begin{equation}\label{eq2.9}
B[u,u] = \int_{\Omega} A\nabla u \cdot \nabla u\,dx + \int_{\partial \Omega}\beta |\mathrm{Tr}(u)|^2\,d\sigma
\ge \mu_0\|\nabla u\|_{L^2( \Omega)}^2 + a_0 \|\mathrm{Tr} (u)\|^2_{L^2(E_0,\sigma)}.
\end{equation}
Moreover, applying Lemma \ref{le2.6}, we conclude that there exists a positive constant $C$ depending only on both the
geometric constants for $(\Omega,\sigma)$ as in Assumption \ref{Ass} and $E_0$ such that, for any $u\in W^{1,2}(\Omega)$,
$$
\int_{\Omega} |u|^2\,dx  \le C \left[ \int_{\Omega} |\nabla u|^2\,dx + \int_{E_0} |\mathrm{Tr} (u)|^2\,d\sigma \right],
$$
which, combined with \eqref{eq2.9}, implies that, for any $u\in W^{1,2}(\Omega)$,
\begin{equation}\label{eq2.10}
B[u,u]\ge\min \{\mu_0, a_0\}\left[\|\nabla u\|_{L^2(\Omega)}^2 +\|u\|_{L^2(E_0,\sigma)}^2\right]\gtrsim \|u\|_{W^{1,2}(\Omega)}^2.
\end{equation}
		
Let $q_1\in(1,\infty)$ satisfy $\frac{1}{q_0} + \frac{2}{q_1} = 1$.  From the assumption that $q_0\in[s/(s+2-n),\infty]$ when $n\ge3$ and
$q_0\in(1,\infty]$ when $n=2$, we deduce that $q_1\in [2,2s/(n-2)]$ when $n\ge3$ and $q_1\in[2,\infty)$ when $n=2$. Then,
by Theorem \ref{th2.1}(iii), we find that, for any $u\in W^{1,2}(\Omega)$,
$\mathrm{Tr}(u)\in L^{q_1}(\partial\Omega,\sigma)$ and $\|\mathrm{Tr}(u)\|_{L^{q_1}(\partial \Omega,\sigma)} \lesssim\|u\|_{W^{1,2}(\Omega)}$,
which, together with H\"older's inequality, further implies that, for any $u,v\in W^{1,2}(\Omega)$,
\begin{align}\label{eq2.11}
|B[u,v]|&\le\left|\int_{\Omega}A\nabla u\cdot\nabla v\,dx\right|+\left| \int_{\partial\Omega}
\beta\mathrm{Tr}(u)\mathrm{Tr}(v)\,d\sigma\right| \notag\\
&\le \mu_0^{-1}\|u\|_{W^{1,2}(\Omega)}\|v\|_{W^{1,2}(\Omega)} + \|\beta\|_{L^{q_0}(\partial\Omega,\sigma)}
\|\mathrm{Tr}(u)\|_{L^{q_1}(\partial \Omega,\sigma)} \|\mathrm{Tr} (v) \|_{L^{q_1}(\partial \Omega,\sigma)}\notag\\
&\lesssim \left[\mu_0^{-1}+\|\beta\|_{L^{q_0}(\partial \Omega,\sigma)}\right] \|u\|_{W^{1,2}(\Omega)} \|v\|_{W^{1,2}(\Omega)}.
\end{align}
		
Let $p_0:=2s/(2s+2-n)$ when $n\ge3$ and $p_0\in(1,\infty)$ when $n=2$. From Theorem \ref{th2.1}(iii), it follows that
$ W^{1,2} (\Omega) \hookrightarrow L^{(p_0)'}(\partial \Omega,\sigma).$
Thus, by this, \eqref{eq2.10}, \eqref{eq2.11}, and the Lax--Milgram theorem (see, for instance, \cite[Theorem 5.8]{GT83}), we conclude that
there exists a unique $u\in W^{1,2}(\Omega)$ such that \eqref{eq2.7} holds for any $\varphi\in C_{\rm c}^\infty(\mathbb{R}^n)$ and \eqref{eq2.8} holds.
This finishes the proof of Theorem \ref{th2.2}.		
\end{proof}

Next, we recall the existence and uniqueness of the weak solution to the partial Neumann boundary value problem (see \cite[Theorem 2.10]{DDEMM24}),
which serves as a key tool in proving the local upper bound estimate of weak solutions to the Robin problem \eqref{Robin-Pro}.
\begin{lemma}\label{le2.7}	
Let $n\ge2$, $s\in(n-2,n)$, and $(\Omega, \sigma)$ satisfy Assumption \ref{Ass} with $\Omega$ being bounded.
Assume that $E\subset \partial \Omega$ is $\sigma$-measurable and $\sigma(E)>0$. Then, for any $g \in L^2(\partial \Omega,\sigma)$, there exists
a unique $u\in W^{1,2}(\Omega)$ with $\mathrm{Tr}(u)|_{E} = 0$ such that, for any $\varphi\in W^{1,2}(\Omega)$ with $\mathrm{Tr}(\varphi)|_{E} = 0$,
\begin{equation*}
\int_{\Omega}A\nabla u \cdot \nabla \varphi\,dx =\int_{\partial \Omega}g \varphi\,d\sigma.
\end{equation*}
\end{lemma}
	
Let $(\Omega,\sigma)$ satisfy Assumption \ref{Ass}. Then we have the following localization lemma (see, for instance, \cite[Lemma 2.6]{DDEMM24}),
which defines a localized domain related to $(\Omega,\sigma)$, called the \emph{tent domain}, that preserves NTA properties of $\Omega$.
	
\begin{lemma}\label{le2.8}
Let $n\ge2$, $s\in(n-2,n)$, and $(\Omega, \sigma)$ satisfy Assumption \ref{Ass} with $\Omega$ being bounded. Then, for any $y \in \partial\Omega$ and
$r \in (0,\mathrm{diam\,}(\Omega)]$, there exists a domain $(T(y,r), \sigma_*)$ satisfying Assumption \ref{Ass} and
\begin{equation*}
B(y, r)\cap \Omega\subset T(y,r)\subset B(y,K_0r)\cap\Omega,
\end{equation*}
where $K_0\in[1,\infty)$ is a constant depending on the geometric constants for $(\Omega,\sigma)$ as in Assumption \ref{Ass}.
Here, the boundary of $T(y,r)$, denoted by $\partial T(y,r)$, is equipped with the measure $\sigma_*$ defined by setting
\begin{equation*}
\sigma_* := \sigma|_{\partial T(y,r)\cap\partial\Omega}+[\delta(\cdot)]^{s-(n-1)}
\mathcal{H}^{n-1}|_{\partial T(y,r)\setminus \partial\Omega},
\end{equation*}
where, for any $z\in\partial T(y,r)\setminus\partial\Omega$, $\delta(z):=\mathrm{dist\,}(z,\partial T(y,r)\cap\partial\Omega)$
and $\mathcal{H}^{n-1}$ denotes the $(n-1)$-dimensional Hausdorff measure.
\end{lemma}

\begin{remark}
Let $(\Omega, \sigma)$ satisfy Assumption \ref{Ass}. It is worth pointing out that $\Omega\cap B(y,r)$ with some $y\in\partial\Omega$
and $r\in(0,\infty)$ may \emph{not} maintain the geometric assumption (A2)
because the \emph{interior corkscrew condition} may not be satisfied near
the boundary $\partial \Omega\cap \partial B(y,r)$. To overcome this problem, Hofmann and Martell \cite[Lemma 3.61]{HM14} constructed
the tent domain $T(y,r)$ when $s=n-1$. Subsequently, Maybroda and Poggi \cite[Sections 4--5]{MP21},
David and Maybroda \cite[Section 9]{DM23}, and David et al. \cite[Lemma 2.6]{DDEMM24} generalized
this result to the general case that $\partial\Omega$ is $s$-Ahlfors regular with $s\in(n-2,n)$.
\end{remark}
	
\section{Proof of Theorem \ref{th1.1}\label{S3}}

In this section, we establish the global H\"older regularity of the weak solution to
the Robin problem \eqref{Robin-Pro} and further give the proof of Theorem \ref{th1.1}.
We first prove the following conclusion.
	
\begin{proposition}\label{pro3.1}
Let $n\ge 2$, $s\in(n-2,n)$, and $(\Omega,\sigma)$ satisfy Assumption \ref{Ass} with $\Omega$ being bounded and let
$\beta$ be the same as in \eqref{eq1.1}.
Assume that $u$ is the weak solution to the Robin problem \eqref{Robin-Pro} with $f\in L^p(\partial\Omega,\sigma)$
for some $p\in (s/(s+2-n),\infty]$.
Then there exists a constant $\alpha_0 \in (0,1]$ depending only on the geometric
constants for $(\Omega,\sigma)$ and $\mu_0$ in \eqref{eq1.2} such that, for any given
$\alpha\in(0,\alpha_0)$, $u\in C^{0,\alpha}(\overline{\Omega})$ and
$
\|u\|_{C^{0,\alpha}(\overline{\Omega})} \le C\|f\|_{L^p(\partial\Omega,\sigma)},
$
where $C$ is a positive constant independent of both $u$ and $f$.
\end{proposition}
	
We prove Proposition \ref{pro3.1} via applying an approach used in the proof of \cite[Theorem 4.2]{V12}.
We begin with recalling some necessary concepts. For any
$x_0\in \partial \Omega$ and $r\in (0,\mathrm{diam\,}(\Omega)]$,
the \emph{Sobolev space} $W^{1,2}_{\rm c}(B(x_0,r)\cap\Omega)$ is defined as the closure of
the space $W(B(x_0,r)\cap\Omega)$ in $W^{1,2}(\Omega)$, where the space $W(B(x_0,r)\cap\Omega)$ is defined
by setting
$$W(B(x_0,r)\cap\Omega):=\{u\phi: u\in W^{1,2}(\Omega), \phi\in C^\infty_{\rm c}(\mathbb{R}^n),
\mathrm{supp\,}(\phi)\subset B(x_0,r)\}. $$
Moreover, for any $u\in W^{1,2}_{\rm c}(B(x_0,r)\cap\Omega)$, define
$
\|u\|_{W^{1,2}_{\rm c}(B(x_0,r)\cap\Omega)}:=\|u\|_{W^{1,2}(\Omega)}.
$
Then it is easy to verify that $W^{1,2}_{\rm c}(B(x_0,r)\cap\Omega)$ is a closed subspace of $W^{1,2}(\Omega)$.
	
Let $\beta$ satisfy \eqref{eq1.1}. Define the new bilinear form $\widetilde{B}[\cdot,\cdot]$ by setting,
for any $u,v\in W^{1,2}_{\rm c}(B(x_0,r)\cap\Omega)$,
\begin{equation*}
\widetilde{B}[u,v] := \int_{B(x_0,r)\cap\Omega} A\nabla u\cdot \nabla v\,dx +
\int_{B(x_0,r)\cap\partial\Omega}\beta \mathrm{Tr}(u) \mathrm{Tr}(v)\,d\sigma.
\end{equation*}
Let $f\in L^p(B(x_0,r)\cap\partial\Omega,\sigma)$ for some $p\in (s/(s+2-n),\infty]$. A function $u\in W^{1,2}_{\rm c}
(B(x_0,r)\cap\Omega)$ is called a \emph{weak solution} to the Robin problem
\begin{equation}\label{eq3.1}
\begin{cases}
-\mathrm{div}(A\nabla u) = 0~~&\text{in}~B(x_0,r)\cap\Omega,\\
A\nabla u\cdot \boldsymbol{\nu} + \beta u = f~~&\text{on}~B(x_0,r)\cap\partial\Omega
\end{cases}
\end{equation}
if, for any $v\in W^{1,2}_{\rm c}(B(x_0,r)\cap\partial\Omega)$,
\begin{equation*}
\widetilde{B}[u,v]=\int_{B(x_0,r)\cap\partial\Omega}f\mathrm{Tr}(v)\,d\sigma.
\end{equation*}

Then we have the following a priori estimate for a weak solution to
the Robin problem \eqref{eq3.1}.
	
\begin{lemma}\label{le3.1}
Let $n\ge 2$, $s\in (n-2,n)$, $(\Omega,\sigma)$ satisfy Assumption \ref{Ass} with $\Omega$ being bounded, and $\beta$ be the same as in \eqref{eq1.1}.
Let $f\in L^p(\partial \Omega,\sigma)$ with $p\in (s/(s+2-n),\infty]$,
$x_0 \in \partial \Omega$, and $r\in (0,\mathrm{diam\,}(\Omega)/4]$.
\begin{itemize}
\item[{\rm (i)}] If $u\in W^{1,2}_{\rm c}(B(x_0,r)\cap\Omega)$ is a weak solution to the Robin problem \eqref{eq3.1},
then there exist constants $\theta, C\in(0,\infty)$ depending only on $n$, $p$, and the geometric constants for $(\Omega,\sigma)$ such that
\begin{equation}\label{eq3.2}
\|u\|_{L^{\infty}(B(x_0,r)\cap\Omega)}+\|\mathrm{Tr}(u)\|_{L^{\infty}(B(x_0,r)\cap\partial\Omega,\sigma)}
\le Cr^{\theta}\|f\|_{L^p(B(x_0,r)\cap\partial\Omega,\sigma)}.
\end{equation}
\item[{\rm (ii)}] If $u\in W^{1,2}(\Omega)$ is the weak solution to the Robin problem \eqref{Robin-Pro},
then there exists a positive constant $C$ depending only on $n$, $p$, and the geometric constants for $(\Omega,\sigma)$ such that
\begin{equation*}
\|u\|_{L^{\infty}(\Omega)}+\|\mathrm{Tr}(u)\|_{L^{\infty}(\partial\Omega,\sigma)}
\le C\|f\|_{L^p(\partial\Omega,\sigma)}.
\end{equation*}
\end{itemize}
\end{lemma}

\begin{proof}
We prove Lemma \ref{le3.1} by borrowing some ideas from the proofs of \cite[Theorem 4.1]{D00} and \cite[Theorem 3.1]{V12}.
Here we only give the proof of (i) because the proof of (ii) is similar.
We now show (i) by considering the following two cases on the spatial dimension $n$.
		
\emph{Case (1)} $n\ge 3$. In this case, by the assumption that $u\in W^{1,2}_{\rm c}(B(x_0,r)\cap\Omega)$ is a weak solution to the
Robin problem \eqref{eq3.1}, we conclude that, for any non-negative function
$v\in W^{1,2}_{\rm c}(B(x_0,r)\cap\Omega)$,
\begin{equation}\label{eq3.3}
\widetilde{B}[u,v]=\int_{B(x_0,r)\cap\partial\Omega} f\mathrm{Tr}(v)\,d\sigma.
\end{equation}
For any given $t\in[1,\infty)$ and $m\in\mathbb{N}$, define the function $\psi_{t,m}:\mathbb{R} \to \mathbb{R}$ by setting,
for any $x\in\mathbb{R}$,
$$\psi_{t,m}(x): = \begin{cases}
0&\quad \text{if}~x\in(-\infty,0], \\
x^t&\quad \text{if}~x\in(0,m),\\
m^{t-1}x&\quad \text{if}~ x\in[m,\infty).
\end{cases}$$
Then it is easy to verify that $\psi_{t,m}\in C(\mathbb{R})$ and is piecewise smooth with bounded derivatives.
Thus, $\psi_{t,m}\circ u \in W_{\rm c}^{1,2}(B(x_0,r)\cap\Omega)$.
		
For any given $k\in [2,\infty)$ and $m\in\mathbb{N}$, let $w_m:= \psi_{k/2,m}(u)$ and $v_m := \psi_{k-1,m} (u)$ for $u$ as in (i). Then
$w_m,v_m \in W_{\rm c}^{1,2}(B(x_0,r)\cap\Omega)$. A direct calculation yields, for any $i,j\in\{1,\ldots,n\}$ and $m\in\mathbb{N}$,
\begin{enumerate}
\item[(i)] when $0\le u \le m$, $\partial_j w_m \partial_i w_m = \frac{k^2}{4(k-1)} \partial_j u
\partial_iv_m$, $w_m \partial_iw_m = \frac{k}{2} v_m\partial_iu= \frac{k}{2(k-1)}
u \partial_i v$, and $w^2_m = u v_m$,
			
\item[(ii)] when $u \ge m$, $\partial_jw_m \partial_i w_m = \partial_j u \partial_i v_m$,
$w_m \partial_i w_m = v_m \partial_i u = u \partial_i v_m,$ and $w_m^2 = u v_m$.
\end{enumerate}
Using this and \eqref{eq1.2}, we conclude that
$$
\|\nabla w_m\|_{L^2(B(x_0,r)\cap\Omega)}^2 \le \mu_0^{-1}\int_{B(x_0,r)\cap\Omega} A\nabla w_m\cdot
\nabla w_m\,dx \le \mu_0^{-1} \gamma_k \int_{B(x_0,r)\cap\Omega}A \nabla u\cdot \nabla v_m\,dx,
$$
where $\gamma_k:= \frac{k^2}{4(k-1)} \ge 1$, which further implies that
\begin{equation}\label{eq3.4}
\|w_m\|^2_{W_{\rm c}^{1,2}(B(x_0,r)\cap\Omega)} \le \mu_0^{-1} \gamma_k \widetilde{B}[u,v_m] +
\|w_m\|_{L^2(B(x_0,r)\cap\Omega)}^2.
\end{equation}
Then, from \eqref{eq3.4}, Lemma \ref{le2.1}, and Theorem \ref{th2.1}(ii) for the case $n\ge 3$, it follows that
\begin{align}\label{eq3.5}
&\|w_m\|^2_{L^{\frac{2n}{n-2}}(B(x_0,r)\cap\Omega)}+\left\|\mathrm{Tr}(w_m)\right\|^2_{L^{\frac{2s}{n-2}}
(B(x_0,r)\cap\partial\Omega,\sigma)} \notag\\
&\quad\lesssim\|w_m\|^2_{W_{\rm c}^{1,2}(B(x_0,r)\cap\Omega)}\lesssim
\mu_0^{-1} \gamma_k \int_{B(x_0,r)\cap\partial\Omega} f\mathrm{Tr}
(v_m)\,d\sigma + \|w_m\|_{L^2(B(x_0,r)\cap\Omega)}^2.
\end{align}
By the definitions of $v_m$ and $w_m$, we find that the pointwise inequality $v_m \le w_m^{2(q-1)/q}$ holds for any $q\in[2,\infty)$,
which, combined with H\"older's inequality, implies that
\begin{align}\label{eq3.6}
\int_{B(x_0,r)\cap\partial\Omega} f\mathrm{Tr}(v_m)\,d\sigma \le & \|f\|_{L^p(B(x_0,r)\cap\partial\Omega,\sigma)}\|\mathrm{Tr}(v_m)\|_{L^{p'}(B(x_0,r)\cap\partial\Omega,\sigma)} \notag\\
\le & \|f\|_{L^p(B(x_0,r)\cap\partial\Omega,\sigma)}\left\|\mathrm{Tr}\left(w_m^{2/k}\right)
\right\|^{k-1}_{L^{p'(k-1)}(B(x_0,r)\cap\partial\Omega,\sigma)},
\end{align}
where $p'\in(1,\infty)$ denotes the conjugate index of $p$.
Let $\zeta\in(1,\infty)$ satisfy $2/\zeta=1/2+s/(np')$. Then $\zeta\in(1,2n/[n-1])$. Take
$p_1\in(\max\{2,\zeta\},2n/[n-2])$. Applying the fact that $\Omega$ is bounded, H\"older's inequality, and the interpolation theorem
(see, for instance, \cite[(7.9)]{GT83}), we obtain
\begin{align}\label{eq3.7}
\|w_m\|_{L^2(B(x_0,r)\cap\Omega)} \lesssim\|w_m\|_{L^{p_1}(B(x_0,r)\cap\Omega)}\lesssim
\|w_m\|^{(1-\theta)}_{L^{\frac{2n}{n-2}}(B(x_0,r)\cap\Omega)} \|w_m\|_{L^{\zeta}(B(x_0,r)\cap\Omega)}^{\theta},
\end{align}
where $\theta\in (0,1)$ is given by
$$
\frac{1}{p_1} = \frac{\theta}{\zeta}+\frac{1-\theta}{\frac{2n}{n-2}}.
$$
Then, from \eqref{eq3.7}, Lemma \ref{le2.1}, and Young's inequality (see, for instance,  \cite[(7.10)]{GT83}), we infer that
there exists a positive constant $C$ such that, for any given $\varepsilon\in(0,\infty)$,
\begin{align}\label{eq3.8}
\|w_m\|_{L^2(B(x_0,r)\cap\Omega)}^2 \le C \varepsilon \|w_m\|_{ W_{\rm c}^{1,2}(B(x_0,r)\cap\Omega) }^2+
C_{(\varepsilon)}\|w_m\|_{L^{\zeta}( B(x_0,r)\cap\Omega)}^2,
\end{align}
where $C_{(\varepsilon)}$ is a positive constant depending on $\varepsilon$.
Let $u^+:=\max\{u,0\}$. Furthermore, by the definition of $w_m$, we conclude that the pointwise inequality $w_m^{2/k}\le u^+$
holds, which, together with H\"older's inequality, further implies that
\begin{align*}
\|w_m\|_{L^{\zeta}( B(x_0,r)\cap\Omega)}^2 &=\left\|w_m^{2/k + 2(k-1)/k} \right\|_{L^{\frac{\zeta}{2}}(B(x_0,r)\cap\Omega)}\\
&\le \left\|w_m^{2/k}\right\|_{L^2( B(x_0,r)\cap\Omega)}\left\|w_m^{2(k-1)/k}\right\|_{L^{\frac{np'}{s}}(B(x_0,r)\cap\Omega)}\\
&\le \|u^+\|_{L^2(B(x_0,r)\cap\Omega)}\left\|w_m^{2/k}\right\|^{k-1}_{L^{\frac{np'(k-1)}{s}}(B(x_0,r)\cap\Omega)}.
\end{align*}
This, combined with \eqref{eq3.8}, \eqref{eq3.5}, and \eqref{eq3.6}, further implies that, for any given $\varepsilon\in(0,\infty)$,
\begin{align}\label{eq3.9}
&\|w_m\|_{L^{2}(B(x_0,r)\cap\Omega)}^2\notag\\
& \quad \le C_{(\varepsilon)}\|u^+\|_{L^{2}( B(x_0,r)\cap\Omega)} \left\|w_m^{2/k} \right\|^{k-1}_{L^{\frac{np'(k-1)}{s}}(B(x_0,r)\cap\Omega)} +
C \varepsilon \|w_m\|^2_{W^{1,2}_{\rm c}( B(x_0,r)\cap\Omega)}\notag\\
&\quad \le  C_{(\varepsilon)}\|u^+\|_{L^{2}(B(x_0,r)\cap\Omega)} \left\|w_m^{2/k}\right
\|^{k-1}_{L^{\frac{np'(k-1)}{s}}(B(x_0,r)\cap\Omega)} \notag\\
&\qquad+ C\varepsilon\left[ \mu_0^{-1} \gamma_k \|f\|_{L^{p}(B(x_0,r)\cap\partial\Omega,\sigma)}
\left\|\mathrm{Tr}\Big(w_m^{2/k}\Big)\right\|^{k-1}_{L^{p'(k-1)}(B(x_0,r)\cap\partial\Omega,\sigma)}
+ \|w_m\|^2_{L^2(B(x_0,r)\cap\Omega)} \right]\notag\\
&\quad \le C_{(\varepsilon)}\gamma_k\left[\|u^+\|_{L^{2}(B(x_0,r)\cap\Omega)} +
\|f\|_{L^{p}(B(x_0,r)\cap\partial\Omega,\sigma)} \right]\notag\\
&\qquad\times\left[ \left\|w_m^{2/k}\right\|^{k-1}_{L^{\frac{np'(k-1)}{s}}(B(x_0,r)\cap\Omega)} + \left\|\mathrm{Tr}\left(w_m^{2/k}\right)
\right\|^{k-1}_{L^{p'(k-1)}(B(x_0,r)\cap\partial\Omega,\sigma)}\right]
+C\varepsilon \|w_m\|^2_{L^2(B(x_0,r)\cap\Omega)}.
\end{align}
Taking $\varepsilon$ small enough such that $C\varepsilon=1/2$ in \eqref{eq3.9} and then applying \eqref{eq3.5},
\eqref{eq3.6}, and \eqref{eq3.9}, we conclude that
\begin{align}\label{eq3.10}
&\|w_m\|^2_{L^{\frac{2n}{n-2}}(B(x_0,r)\cap\Omega)}+\left\|\mathrm{Tr}(w_m)
\right\|^2_{L^{\frac{2s}{n-2}}(B(x_0,r)\cap\partial\Omega,\sigma)} \notag\\
&\quad = \left\|w_m^{2/k}\right\|^k_{L^{\frac{nk}{n-2}}(B(x_0,r)\cap\Omega)} + \left\|\mathrm{Tr}\left(w_m^{2/k}\right)
\right\|^k_{L^{\frac{sk}{n-2}}(B(x_0,r)\cap\partial\Omega,\sigma)} \notag\\
&\quad \le C \gamma_k \left[\|f\|_{L^p(B(x_0,r)\cap\partial\Omega,\sigma)} + \|u^+\|_{L^2(B(x_0,r)\cap\Omega)}\right] \notag\\
&\qquad\times \left[\left\|w_m^{2/k}\right\|^{k-1}_{L^{\frac{np'(k-1)}{s}}(B(x_0,r)\cap\Omega)} + \left\|\mathrm{Tr}\left(w_m^{2/k}\right)
\right\|^{k-1}_{L^{p'(k-1)}(B(x_0,r)\cap\partial\Omega,\sigma)} \right].
\end{align}
Let
$$M:=C\left[\|f\|_{L^p(B(x_0,r)\cap\partial\Omega,\sigma)}+\|u^+\|_{L^2(B(x_0,r)\cap\Omega)}\right].$$
We first assume that $M\in(0,\infty)$.
Taking $\widetilde{u}:= u^+/M$ and letting $m\to \infty$
in \eqref{eq3.10}, and then applying the monotone convergence theorem, we further obtain
\begin{align}\label{eq3.11}
&\|\widetilde{u}\|^k_{L^{\frac{nk}{n-2}}(B(x_0,r)\cap\Omega)} + \|\mathrm{Tr}(\widetilde{u})\|^k_{L^{\frac{sk}{n-2}}
(B(x_0,r)\cap\partial\Omega,\sigma)} \notag \\
&\quad \le k^2\left[\|\widetilde{u}\|^{k-1}_{L^{\frac{np'(k-1)}{s}}(B(x_0,r)\cap\Omega)}+\|\mathrm{Tr}(\widetilde{u})
\|^{k-1}_{L^{p'(k-1)}(B(x_0,r)\cap\partial\Omega,\sigma)}\right].
\end{align}
		
Next, we devise an iterative scheme. Let
$\kappa:= \frac{s}{(n-2)p'}= \frac{s(p-1)}{(n-2)p}.$
Since $p\in(s/(s+2-n),\infty]$, it follows that $\kappa\in(1,\infty)$. Define the increasing sequence $\{k_i\}_{i=0}^\infty$ in $[2,\infty)$
by setting $k_0:=2$ and, for any $i\in\mathbb{Z}_+$,
$k_{i+1}:=1+\kappa k_i.$
Then, applying induction to \eqref{eq3.11} and using the fact that $p'\kappa k_i \le \frac{sk_i}{n-2}$,
we conclude that, for any $i\in\mathbb{Z}_+$,
\begin{align}\label{eq3.12}
&\|\widetilde{u}\|^{k_{i+1}}_{L^{\frac{nk_{i+1}}{n-2}} (B(x_0,r)\cap\Omega)}+
\|\mathrm{Tr}(\widetilde{u})\|^{k_{i+1}}_{L^{\frac{sk_{i+1}}{n-2}} (B(x_0,r)\cap\partial\Omega,\sigma)} \notag \\
&\quad \le  k_{i+1}^2 \left[\|\widetilde{u}\|^{{\kappa k_i}}_{L^{\frac{np'\kappa k_i}{s}}(B(x_0,r)\cap\Omega)}
+\|\mathrm{Tr}(\widetilde{u})\|^{\kappa k_i}_{L^{p'\kappa k_i}(B(x_0,r)\cap\partial\Omega,\sigma)} \right] \notag \\
&\quad \le Ck_{i+1}^2 \left[\|\widetilde{u}\|^{{\kappa k_i}}_{L^{\frac{n \kappa k_i}{n-2}} ( B(x_0,r)\cap\Omega)}
+ \|\mathrm{Tr}(\widetilde{u})\|^{\kappa k_i}_{L^{\frac{s\kappa k_i}{n-1}}(B(x_0,r)\cap\partial\Omega,\sigma)} \right] \notag \\
&\quad \le C\prod\limits_{j=1}^{i+1} k_j ^{2\kappa^{i-j+1}} \left[\|\widetilde{u}\|^{2\kappa^{i+1}}_{L^{\frac{2n}{n-2}}(B(x_0,r)\cap\Omega)}
+ \|\mathrm{Tr}(\widetilde{u})\|^{2\kappa^{i+1}}_{L^{\frac{2s}{n-2}}(B(x_0,r)\cap\partial\Omega,\sigma)} \right].
\end{align}
Notice that, for any $i\in\mathbb{Z}_+$, $\kappa \le
\frac{k_{i+1}}{k_i} \le 2\kappa$. This further yields that,
for any $i\in\mathbb{Z}_+$, $\kappa^i \le k_i \le (2\kappa)^i$ and
$ k_{i+1}=\kappa^{i+1}+\sum_{j=0}^{i+1}\kappa^j.$
Thus, from this and \eqref{eq3.12}, we deduce that, for any $i\in\mathbb{Z}_+$,
\begin{align}\label{eq3.13}
&\left[\|\widetilde{u}\|^{k_{i+1}}_{L^{\frac{n{k_{i+1}}}{n-2}}(B(x_0,r)\cap\Omega)}
+\|\mathrm{Tr}(\widetilde{u})\|^{k_{i+1}}_{L^{\frac{s{k_{i+1}}}{n-2}}
(B(x_0,r)\cap\partial\Omega,\sigma)}\right]^{\frac1{k_{i+1}}} \notag \\
&\quad \le C(2\kappa)^{2\sum_{j=1}^{i+1} j\kappa^{-j}} \left[\|\widetilde{u}\|^{2(1+\sum_{j=0}^{i+1}
\kappa^{-j})^{-1}}_{L^{\frac{2n}{n-2}}(B(x_0,r)\cap\Omega)} + \|\mathrm{Tr}(\widetilde{u})\|^{2(1+ \sum_{j=0}^{i+1}
\kappa^{-j})^{-1}}_{L^{\frac{2s}{n-2}}(B(x_0,r)\cap\partial\Omega,\sigma)} \right].
\end{align}
Using the fact that $\kappa\in(1,\infty)$ and letting $i\to\infty$ in \eqref{eq3.13}, we find that $\widetilde{u}\in L^\infty(B(x_0,r)\cap\Omega)$,
$\mathrm{Tr}(\widetilde{u})\in L^{\infty}(B(x_0,r)\cap\partial\Omega,\sigma)$, and
$$
\|\widetilde{u}\|_{L^{\infty}(B(x_0,r)\cap\Omega)} + \|\mathrm{Tr}(\widetilde{u})\|_{L^{\infty}(B(x_0,r)\cap\partial\Omega,\sigma)}
\le C_{\kappa}\left[\|\widetilde{u}\|^{\delta_\kappa}_{L^{\frac{2n}{n-2}}(B(x_0,r)\cap\Omega)}
+\|\mathrm{Tr}(\widetilde{u})
\|^{\delta_\kappa}_{L^{\frac{2s}{n-2}}(B(x_0,r)\cap\partial\Omega,\sigma)}\right],
$$
where $C_{\kappa}:=C(2\kappa)^{2\sum_{j=1}^{\infty}j\kappa^{-j} }<\infty$
and $\delta_\kappa:= 1-\frac{1}{2\kappa -1}\in (0,1)$, which further implies that
\begin{align}\label{eq3.14}
&\|\widetilde{u}\|_{L^{\infty}(B(x_0,r)\cap\Omega)}+\|\mathrm{Tr}(\widetilde{u})\|_{L^{\infty}(B(x_0,r)\cap\partial\Omega,\sigma)}\notag \\
&\quad\le C_{\kappa}\left\{|B(x_0,r)\cap\Omega|^{\frac{(n-2)\delta_\kappa}{2n}}
\|\widetilde{u}\|_{L^{\infty}(B(x_0,r)\cap\Omega)}^{\delta_\kappa}\right.\notag \\
&\qquad\left.+ [\sigma(B(x_0,r)\cap\partial\Omega)]^{\frac{(n-2)\delta_\kappa}{2s}} \|\mathrm{Tr}(\widetilde{u})\|_{L^{\infty}(B(x_0,r)\cap\partial\Omega,\sigma)}^{\delta_\kappa} \right\}.
\end{align}
Thus, by \eqref{eq3.14}, we conclude that
\begin{align}\label{eq3.15}
&\|\widetilde{u}\|_{L^{\infty}(B(x_0,r)\cap\Omega)} + \|\mathrm{Tr}(\widetilde{u})\|_{L^{\infty}(B(x_0,r)\cap\partial\Omega,\sigma)} \notag \\
&\quad \le C \left\{|B(x_0,r)\cap\Omega|^{\frac{n-2}{2n}\frac{\delta_\kappa}{1-\delta_\kappa}}
+ [\sigma(B(x_0,r)\cap\partial\Omega)]^{\frac{n-2}{2s}
\frac{\delta_\kappa}{1-\delta_\kappa}}\right\} \le C r^{\frac{n-2}{2}\frac{\delta_\kappa}{1-\delta_\kappa}}.
\end{align}
Let $\theta:= \frac{n-2}{2}\frac{\delta_\kappa}{1-\delta_\kappa}$. From \eqref{eq3.15} and the definition of $\widetilde{u}$,
it follows that
\begin{align}\label{eq3.16}
&\|u^+\|_{L^{\infty}(B(x_0,r)\cap\Omega)} + \|\mathrm{Tr}(u^+)\|_{L^{\infty}(B(x_0,r)\cap\partial\Omega,\sigma)}\lesssim r^{\theta}
\left[\|f\|_{L^p(B(x_0,r)\cap\partial\Omega,\sigma)}+\|u^+\|_{L^2(B(x_0,r)\cap\Omega)}\right].
\end{align}

Moreover, if $M=0$, by the definition of $M$, we find that $\|u^+\|_{L^2(B(x_0,r)\cap\Omega)}=0$ in this case, which,
together with the definition of $\mathrm{Tr}(u^+)$, further implies that, in this case,
$\|u^+\|_{L^{\infty}(B(x_0,r)\cap\Omega)}=0=\|\mathrm{Tr}(u^+)
\|_{L^{\infty}(B(x_0,r)\cap\partial\Omega,\sigma)}$.
Thus, when $M=0$, \eqref{eq3.16} also holds.

Let $u^{-}:=\max\{-u,0\}$. Replacing $u$ and $f$ in \eqref{eq3.3} respectively by $-u$ and $-f$, and repeating
the proof of \eqref{eq3.16}, we obtain
\begin{align*}
&\|u^-\|_{L^{\infty}(B(x_0,r)\cap\Omega)} + \|\mathrm{Tr}(u^-)\|_{L^{\infty}(B(x_0,r)\cap\partial\Omega,\sigma)}\lesssim r^{\theta}
\left[\|f\|_{L^p(B(x_0,r)\cap\partial\Omega,\sigma)}+\|u^-\|_{L^2(B(x_0,r)\cap\Omega)}\right],
\end{align*}
which, combined with \eqref{eq3.16}, implies that
\begin{align}\label{eq3.17}
\|u\|_{L^{\infty}(B(x_0,r)\cap\Omega)} + \|\mathrm{Tr}(u)\|_{L^{\infty}(B(x_0,r)\cap\partial\Omega,\sigma)}\lesssim r^{\theta}
\left[\|f\|_{L^p(B(x_0,r)\cap\partial\Omega,\sigma)}+\|u\|_{L^2(B(x_0,r)\cap\Omega)}\right].
\end{align}

Furthermore, from $u\in W^{1,2}_{\rm c}(B(x_0,r)\cap\Omega)$ and the definition of the space $W^{1,2}_{\rm c}(B(x_0,r)\cap\Omega)$,
we deduce that $\mathrm{Tr}(u)|_{\partial B(x_0,r)\cap\Omega}=0$, which, combined with Lemma \ref{le2.5}, further implies that
$\|u\|_{W^{1,2}_{\rm c}(B(x_0,r)\cap\Omega)}\sim\|\nabla u\|_{L^2(B(x_0,r)\cap\Omega))}$.
By this, \eqref{eq3.3}, $p>s/(s+2-n)\ge2s/(2s+2-n)$, and Theorem \ref{th2.1}(iii), we find that
\begin{align*}
\|u\|^2_{W^{1,2}_{\rm c}(B(x_0,r)\cap\Omega)}&\lesssim\|\nabla u\|^2_{L^2(B(x_0,r)\cap\Omega))}
\lesssim\|f\|_{L^p(B(x_0,r)\cap\partial\Omega,\sigma)}
\|\mathrm{Tr}(u)\|_{L^{p'}(\partial\Omega,\sigma)}\\
&\lesssim\|f\|_{L^p(B(x_0,r)\cap\partial\Omega,\sigma)}\|u\|_{W^{1,2}_{\rm c}(B(x_0,r)\cap\Omega)},
\end{align*}
which further implies that
$$
\|u\|_{L^2(B(x_0,r)\cap\Omega)}\le\|u\|_{W^{1,2}_{\rm c}(B(x_0,r)\cap\Omega)}\lesssim\|f\|_{L^p(B(x_0,r)\cap\partial\Omega,\sigma)}.
$$
This, together with \eqref{eq3.17}, implies that
\begin{align*}
\|u\|_{L^{\infty}(B(x_0,r)\cap\Omega)}+\|\mathrm{Tr}(u)\|_{L^{\infty}
(B(x_0,r)\cap\partial\Omega,\sigma)}\lesssim r^{\theta}
\|f\|_{L^p(B(x_0,r)\cap\partial\Omega,\sigma)}.
\end{align*}
Therefore, \eqref{eq3.2} holds in the case $n\ge3$.

\emph{Case (2)} $n = 2$. In this case, using the conclusions of Lemma \ref{le2.1} and Theorem \ref{th2.1}(iii)
in the case of $n=2$, similar to the proof of \eqref{eq3.11}, we conclude that, for any given $\widetilde{p}_1,\widetilde{p}_2\in(p',\infty)$,
\begin{align}\label{eq3.18}
&\|\widetilde{u}\|^k_{L^{\widetilde{p}_1k} (B(x_0,r)\cap\Omega)} +
\|\mathrm{Tr}(\widetilde{u})\|^k_{L^{\widetilde{p}_2k} (B(x_0,r)\cap\partial\Omega,\sigma)} \notag \\
&\quad \le k^2\left[\|\widetilde{u}\|^{k-1}_{L^{p'(k-1)}( B(x_0,r)\cap\Omega)}+\|\mathrm{Tr}(\widetilde{u})\|^{k-1}_{L^{p'(k-1)}
(B(x_0,r)\cap\partial\Omega,\sigma)}\right].
\end{align}
Then, repeating the proof of Case (1) with \eqref{eq3.11} replaced by \eqref{eq3.18},
we further find that \eqref{eq3.2} also holds in this case. This
finishes the proof of Lemma \ref{le3.1}.
\end{proof}

Furthermore, similar to the proof of \cite[Lemma 4.1]{V12} (see also \cite[Example 6.6]{V12}),
we obtain the following oscillation estimates for weak solutions to the Robin problem \eqref{Robin-Pro} with $f\equiv0$
on $B(x_0,r)\cap\partial\Omega$ for any given $x_0\in\partial\Omega$ and $r\in(0,\mathrm{diam\,}(\Omega))$.
	
\begin{lemma}\label{le3.2}
Let $n\ge2$, $s\in(n-2, n)$, and $(\Omega,\sigma)$ satisfy Assumption \ref{Ass} with $\Omega$ being bounded.
Let $\beta$ be the same as in \eqref{eq1.1},
$x_0 \in \partial \Omega$, and $r\in (0,\mathrm{diam}\,(\Omega))$. Assume that $u\in W^{1,2}(\Omega)$
is bounded on $\overline{\Omega}$ and satisfies
that, for any $\varphi\in W^{1,2}_{\rm c}(B(x_0,r)\cap\Omega)$,
\begin{equation}\label{eq3.19}
B[u,\varphi]=0.
\end{equation}
Then there exists a constant $\eta\in (0,1)$, independent of $u$, $x_0$,
and $r$, such that
\begin{equation}\label{eq3.20}
\mathop{\mathrm{osc}}_{B(x_0,r/4)\cap\overline{\Omega}}
u\le \eta \mathop{\mathrm{osc}}_{B(x_0,r)\cap\overline{\Omega}}u,
\end{equation}
where, for any measurable set $E\subset\overline{\Omega}$, 
$$\mathop{\mathrm{osc}}_{E}u:=
\mathop{\mathrm{ess\,sup}}_{x\in E}u(x)-
\mathop{\mathrm{ess\,inf}}_{x\in E}u(x).$$
\end{lemma}

\begin{proof}
We prove the present lemma by borrowing some ideas from the proof of \cite[Lemma 4.1]{V12}.

For any given $k\in(0,\infty)$ and $\rho\in(0,r)$, let
$$A(k,\rho):= \{x\in B(x_0,\rho)\cap\Omega:u(x)>k\}.$$
Then, similar to the proof of \cite[Proposition 3.2]{V12},
we conclude that, for any $k\in(0,\infty)$ and $0<\rho_1<\rho_2<r$,
\begin{equation}\label{eq3.21}
\int_{A(k,\rho_1)}|\nabla u|^2\,dx\lesssim (\rho_2-\rho_1)^{-2}
\int_{A(k,\rho_2)}|u-k|^2\,dx.
\end{equation}
Moreover, replacing \cite[(2.2) and Proposition 2.1]{V12} respectively
by Theorem \ref{th2.1}(iii) and Lemma \ref{le2.3}, and repeating
the proof of \cite[Proposition 3.4]{V12}, we find that, for any $k\in(0,\infty)$
and $\rho_1\in(0,r)$ satisfying $|A(k,\rho_1)|\le|B(x_0,\rho_1)\cap\Omega|/2$ and for any $\rho_2\in(\rho_1,r)$ and $h\in(k,\infty)$,
\begin{equation}\label{eq3.22}
\sigma(\partial\Omega\cap A(h,\rho_1))^{\frac{2}{\kappa_0}}
\lesssim(h-k)^{-2}(\rho_2-\rho_1)^{-2}
\int_{A(k,\rho_2)}|u-k|^2\,dx,
\end{equation}
where $\kappa_0:=2s/(n-2)$ when $n\ge3$ and $\kappa_0\in(1,\infty)$ is
a given constant when $n=2$.

Then, repeating the proof of \cite[Lemma 4.1]{V12} with
\cite[Propositions 3.2 and 3.4]{V12} replaced respectively
by \eqref{eq3.21} and \eqref{eq3.22}, we conclude that \eqref{eq3.20}
holds. This then finishes the proof of Lemma \ref{le3.2}.
\end{proof}
	
Now, we show Proposition \ref{pro3.1} by using Lemmas \ref{le3.1} and \ref{le3.2}.
	
\begin{proof}[Proof of Proposition \ref{pro3.1}]
Let $u \in W^{1,2}(\Omega)$ be the weak solution to the Robin problem \eqref{Robin-Pro} with $f\in L^p(\partial\Omega,\sigma)$
for some $p\in (s/(s+2-n),\infty]$.
Then, for any $\varphi \in W^{1,2}(\Omega)$,
$$
B[u,\varphi] = \int_{\partial \Omega} f\mathrm{Tr} (\varphi)\,d\sigma.
$$
By Lemma \ref{le3.1}, we find that $u$ is bounded on $\overline{\Omega}$. Let $x_0 \in \partial \Omega$, $r \in (0,\mathrm{diam\,}(\Omega)/4]$, and
$\widetilde{u} \in W^{1,2}_{\rm c}(B(x_0,r)\cap\Omega)$ be a weak solution to the Robin problem \eqref{eq3.1}.
From Lemma \ref{le3.1}, we deduce that there exists a positive constant
$\theta$ such that
\begin{equation}\label{eq3.23}
\left\|\widetilde{u}\right\|_{L^{\infty}{(B(x_0,r)\cap\Omega)}}
+\left\|\mathrm{Tr}(\widetilde{u})\right\|_{{L^{\infty}{(B(x_0,r)
\cap\partial\Omega,\sigma)}}} \lesssim r^{\theta}\|f\|_{L^p(B(x_0,r)\cap\partial\Omega,\sigma)}.
\end{equation}
Let $v:= u-\widetilde{u}$. Then $v \in W^{1,2}(\Omega)$ is bounded on $\overline{\Omega}$ and satisfies \eqref{eq3.19}.
Then, by Lemma \ref{le3.2}, we conclude that there exists a constant $\eta \in (0,1)$
independent of $x_0$ and $r$ such that
$$
\mathop{\mathrm{osc}}_{B(x_0,r/4)\cap\overline{\Omega}}v\le \eta \mathop{\mathrm{osc}}_{B(x_0,r)\cap\overline{\Omega}}v,
$$
which, combined with $v= u-\widetilde{u}$ and the definition of $\mathop{\mathrm{osc}}_{B(x_0,r/4)\cap\overline{\Omega}}v$, further implies that
\begin{align}\label{eq3.24}
\mathop{\mathrm{osc}}_{B(x_0,r/4)\cap\overline{\Omega}}u &\le \mathop{\mathrm{osc}}_{B(x_0,r/4)\cap\overline{\Omega}}\widetilde{u}
+\mathop{\mathrm{osc}}_{B(x_0,r/4)\cap\overline{\Omega}}v \notag \\
&\le 2\left[\|\widetilde{u}\|_{L^{\infty}{(B(x_0,r)\cap\Omega)}}
+\|\mathrm{Tr}(\widetilde{u})\|_{L^{\infty}(B(x_0,r)\cap\partial\Omega,\sigma)}\right]
+\eta\mathop{\mathrm{osc}}_{B(x_0,r)\cap\overline{\Omega}}v\notag \\
&\le 2(1+\eta)\left[\|\widetilde{u}\|_{L^{\infty}{(B(x_0,r)\cap\Omega)}}+\|\mathrm{Tr}(\widetilde{u})
\|_{L^{\infty}(B(x_0,r)\cap\partial\Omega,\sigma)}\right]+ \eta \mathop{\mathrm{osc}}_{B(x_0,r)\cap\overline{\Omega}}u.
\end{align}
From \eqref{eq3.23} and \eqref{eq3.24}, we infer that, for any $r\in(0,\mathrm{diam\,}(\Omega)/4]$,
\begin{align}\label{eq3.25}
\mathop{\mathrm{osc}}_{B(x_0,r/4)\cap\overline{\Omega}}u &\le \eta \mathop{\mathrm{osc}}_{B(x_0,r)\cap\overline{\Omega}}u
+Cr^{\theta} \|f\|_{L^p(B(x_0,r)\cap\partial\Omega,\sigma)}\notag \\
&\le \eta \mathop{\mathrm{osc}}_{B(x_0,r)\cap\overline{\Omega}}{\mathrm{osc}}\, u
+Cr^{\theta}\|f\|_{L^p(\partial\Omega,\sigma)},
\end{align}
where $C$ is a positive constant independent of $f$, $u$, $r$, $\eta$, and $\theta$.
Then, by \eqref{eq3.25} and \cite[Lemma 4.19]{hl11}, we find that, for any given $\theta_0\in(0,\theta)$,
there exists a constant
$\alpha_0\in(0,1]$, depending only on $\eta$, $\theta_0$, and the geometric constants for $(\Omega,\sigma)$, such that, for any
$r\in(0,r_0]$ with $r_0:=\mathrm{diam\,}(\Omega)/4$,
\begin{align}\label{eq3.26}
\mathop{\mathrm{osc}}_{B(x_0,r)\cap\overline{\Omega}}\, u \lesssim \left(\frac{r}{r_0}\right)^{\alpha_0}
\mathop{\mathrm{osc}}_{B(x_0,r_0)\cap\overline{\Omega}}\, u
+r^{\theta_0}\|f\|_{L^p(\partial\Omega,\sigma)}.
\end{align}
	
Next, we estimate the H\"older semi-norm of $u$ on $B(x_0,r_0)\cap\overline{\Omega}$. From
\eqref{eq3.26} and Lemma \ref{le3.1}(ii), it follows that, for any $x,y\in B(x_0,r_0)\cap\overline{\Omega}$,
\begin{align*}
\frac{|u(x)-u(y)|}{|x-y|^{\alpha_0}} &\lesssim
\frac{1}{r_0^{\alpha_0}}
\mathop{\mathrm{osc}}_{B(x_0,r_0)\cap\overline{\Omega}}u +
\|f\|_{L^p(\partial\Omega,\sigma)}\notag \\
&\lesssim r_0^{-\alpha_0}\left[\|u\|_{L^{\infty}(B(x_0,r_0)\cap\Omega)}+\|\mathrm{Tr}(u)
\|_{L^{\infty}(B(x_0,r_0)\cap\partial\Omega,\sigma)}\right] + \|f\|_{L^p(\partial\Omega,\sigma)}\notag \\
&\lesssim \|f\|_{L^p(\partial\Omega,\sigma)},
\end{align*}
which, together with Lemma \ref{le3.1}(ii) again, further implies that, for any given $\alpha\in(0,\alpha_0]$,
\begin{align}\label{eq3.27}
\|u\|_{C^{0,\alpha}(B(x_0,r_0)\cap\overline{\Omega})}& = \|u\|_{L^{\infty}(B(x_0,r_0)\cap\overline{\Omega})} +
\sup\limits_{\substack{x,y\in B(x_0,r_0)\cap\overline{\Omega}}}  \frac{|u(x)-u(y)|}{|x-y|^\alpha}\notag \\
 &\lesssim \|f\|_{L^p(\partial\Omega,\sigma)}.
\end{align}
Moreover, by the well-known interior H\"older estimate for $u$ (see, for instance, \cite[Lemma 11.30]{DFM23}),
we conclude that, for any $x_1\in\Omega$ and $R_0\in(0,\mathrm{diam\,}(\Omega)/4]$ satisfying $B(x_1,2R_0)\subset\Omega$,
there exists a constant $\alpha_1\in(0,1]$, depending only on the geometric constants for $(\Omega,\sigma)$
and $\mu_0$ in \eqref{eq1.2}, such that, for any $r\in(0,R_0)$,
$$
\mathop{\mathrm{osc}}_{B(x_1,r)}\, u \lesssim \left(\frac{r}{R_0}\right)^{\alpha_1}\left[\fint_{B(x_1,R_0)}
|u|^2\,dy\right]^{\frac12},
$$
which, combined with Lemma \ref{le3.1}(ii), further implies that, for any given $\alpha\in(0,\alpha_1]$,
\begin{align}\label{eq3.28}
\|u\|_{C^{0,\alpha}(B(x_1,R_0))}\lesssim R_0^{-\alpha}\|f\|_{L^p(\partial\Omega,\sigma)}.
\end{align}
Therefore, from the arbitrariness of $x_0\in\partial\Omega$ in \eqref{eq3.27} and $x_1\in\Omega$ in \eqref{eq3.28},
\eqref{eq3.27}, \eqref{eq3.28}, and a standard covering argument, we further deduce that
$u\in C^{0,\alpha}(\overline{\Omega})$ with any give $\alpha\in(0,\min\{\alpha_0,\alpha_1\}]$
and $\|u\|_{C^{0,\alpha}(\overline{\Omega})}\lesssim \|f\|_{L^p(\partial\Omega,\sigma)}$.
This finishes the proof of Proposition \ref{pro3.1}.
\end{proof}
	
Finally, we give the proof of Theorem \ref{th1.1} by using Theorem \ref{th2.2} and Proposition \ref{pro3.1}.
\begin{proof}[Proof of Theorem \ref{th1.1}]
By the assumptions $n\ge2$ and $s\in(n-2,n)$, we find that $s/(s+2-n)\ge2s/(2s+2-n)$.
From this, the assumption that $\Omega$ is bounded, and Theorem \ref{th2.2}, we deduce that
there exists a unique weak solution to the Robin problem \eqref{Robin-Pro} with $f\in L^p(\partial \Omega,\sigma)$
for some $p\in (s/(s+2-n),\infty]$.
		
Moreover, by Proposition \ref{pro3.1}, we conclude that there exists $\alpha_0\in(0,1]$
such that, for any given $\alpha\in(0,\alpha_0)$, $u\in C^{0,\alpha}(\overline{\Omega})$ and \eqref{eq1.5} holds.
This finishes the proof of Theorem \ref{th1.1}.
\end{proof}

\section{Proof of Theorem \ref{th1.2}\label{S5}}
In this section, we prove Theorem \ref{th1.2}. We begin with establishing the following Moser type estimate for the Robin problem \eqref{Robin-Pro}.

\begin{lemma}\label{le4.1}
Let $n\ge 2,\ s\in (n-2,n)$, and $(\Omega,\sigma)$ satisfy Assumption \ref{Ass} with $\Omega$
being bounded. Assume that the constant $K_0\in [1,\infty)$
is as in Lemma \ref{le2.8}, $x_0\in\partial\Omega$, and $r\in(0,\mathrm{diam\,}(\Omega)/4]$. Let the non-negative function
$\tau\in L^q(B(x_0,2K_0r)\cap\partial\Omega,\sigma)$ for some $q\in[s/(s+2-n),\infty]$ when $n\ge3$ and some $q\in(1,\infty]$ when $n=2$.
Assume that $0\le w\in W^{1,2}(B(x_0,2K_0r)\cap\Omega)$
satisfies that, for any $\varphi \in W^{1,2}(B(x_0,2K_0r)\cap\Omega)$ satisfying $\varphi\ge 0$
and $\varphi\equiv0$ on $\Omega\backslash B(x_0,\rho)$ for some
$\rho\in(0,2K_0r)$,
\begin{equation}\label{eq4.1}
\int_{B(x_0,2K_0r)\cap\Omega} A\nabla w\cdot \nabla \varphi\,dx \le \int_{B(x_0,2K_0r)\cap\partial\Omega}\tau \varphi\,d\sigma.
\end{equation}
Then there exists a positive constant $C$ depending only on the geometric constants for
$(\Omega,\sigma)$ and $\mu_0$ in \eqref{eq1.2} such that
\begin{equation}\label{eq4.2}
\sup\limits_{x\in B(x_0,r)\cap\Omega} w (x) \le C\left\{\|\tau\|_{L^q(B(x_0,2r)\cap\partial\Omega,\sigma)}+
\left[\fint_{B(x_0,2r)\cap\Omega} w^2\,dx\right]^{\frac12}\right\}.
\end{equation}
\end{lemma}
\begin{proof}
We prove the present lemma by borrowing some ideas from the proof of \cite[Lemma 3.2]{DDEMM24}.
Let $M\in(0,\infty)$ be a constant large enough. For any $x\in B(x_0,2K_0r)\cap\Omega$, let
$$v(x):=\begin{cases}
w(x)+ k&\text{if} \ 0<w(x)<M,\\
M+k&\text{otherwise},
\end{cases} $$
where the constant $k\in(0,\infty)$ will be chosen later. Then, for any $x\in B(x_0,2K_0r)\cap\Omega$, $k<v(x)\le M+k$.
		
Take $\varphi \in C_{\rm c}^{\infty}(B(x_0,2r))$ satisfy $0\le\varphi\le1$. For any given $\alpha\in[1,\infty)$, define
$h:= v^{\alpha-1}(w+k)\varphi^2.$
Since $v$ is bounded and $\nabla v =0$ when $w\ge M$, it follows that $h\in W^{1,2}(B(x_0,2K_0r)\cap\Omega)$ and
\begin{equation}\label{eq4.3}
\nabla h=(\alpha-1)\left(v^{\alpha-2}\nabla v\right)(w+k)\varphi^2+v^{\alpha-1} \varphi^2\nabla w + 2\varphi v^{\alpha-1}
(w+k) \nabla \varphi \in L^2(B(x_0,2K_0r)\cap\Omega).
\end{equation}
Thus, using $h$ as a test function in \eqref{eq4.1}, we find that
\begin{equation}\label{eq4.4}
\int_{B(x_0,2r)\cap\Omega} A\nabla w\cdot \nabla h\,dx \le \int_{B(x_0,2r)\cap\partial\Omega} \tau h\,d\sigma
= \int_{B(x_0,2r)\cap\partial\Omega}\tau v^{\alpha-1} (w+k) \varphi^2\,d\sigma.
\end{equation}
Moreover, notice that $A\nabla w\cdot\nabla v=A\nabla v\cdot\nabla v$. From this and \eqref{eq4.3}, we deduce that
\begin{align}\label{eq4.5}
&A\nabla w \cdot \nabla h - A\nabla w\cdot [2\varphi v^{\alpha-1}(w+k) \nabla \varphi] \notag \\
& \quad= (\alpha-1) v^{\alpha-2} (w+k) \varphi^2 A\nabla v\cdot \nabla v +v^{\alpha-1} \varphi^2 A\nabla w\cdot \nabla w \notag \\
&\quad \ge  (\alpha-1) v^{\alpha-1} \varphi^2 A\nabla v\cdot \nabla v + v^{\alpha-1} \varphi^2 A\nabla w\cdot \nabla w.
\end{align}
Then, by \eqref{eq1.2}, \eqref{eq4.4}, and \eqref{eq4.5}, we further conclude that
\begin{align}\label{eq4.6}
\mathrm{I}&:=(\alpha-1)\int_{B(x_0,2r)\cap\Omega}v^{\alpha-1}\varphi^2|\nabla v|^2\,dx+\int_{B(x_0,2r)\cap\Omega}
v^{\alpha-1}\varphi^2|\nabla w|^2\,dx\notag \\
&\le C(\alpha-1)\int_{B(x_0,2r)\cap\Omega}v^{\alpha-1}\varphi^2 A\nabla v\cdot\nabla v\,dx
+C\int_{B(x_0,2r)\cap\Omega}v^{\alpha-1}\varphi^2A\nabla w\cdot\nabla w\,dx\notag \\
&\le C\int_{B(x_0,2r)\cap\Omega} \varphi v^{\alpha-1} (w+k) |\nabla w| |\nabla \varphi|\,dx
+C\int_{B(x_0,2r)\cap\partial\Omega}\tau v^{\alpha-1}(w+k)\varphi^2\,d\sigma,
\end{align}
where $C$ is a positive constant depending only on $\mu_0$ in \eqref{eq1.2}.
From H\"older's inequality and Young's inequality, we infer that, for any given $\varepsilon\in(0,1)$,
\begin{align*}
&\int_{B(x_0,2r)\cap\Omega} \varphi v^{\alpha-1} (w+k) |\nabla w| |\nabla \varphi|\,dx\\
&\quad\le \left[\int_{B(x_0,2r)\cap\Omega} \varphi^2 v^{\alpha-1} |\nabla w|^2\,dx \right]^{\frac12}
\left[\int_{B(x_0,2r)\cap\Omega} v^{\alpha-1} (w+k)^2 |\nabla \varphi|^2\,dx \right]^{\frac12} \\
&\quad\le \varepsilon \int_{B(x_0,2r)\cap\Omega} \varphi^2 v^{\alpha-1} |\nabla w|^2\,dx
+ C_{\varepsilon}\int_{B(x_0,2r)\cap\Omega} v^{\alpha-1} (w+k)^2 |\nabla \varphi|^2 \,dx
\end{align*}
with $C_{\varepsilon}$ being a positive constant depending only
on $\varepsilon$, which, combined with \eqref{eq4.6}, further implies that
\begin{equation}\label{eq4.7}
\mathrm{I}\le C\int_{B(x_0,2r)\cap\Omega} v^{\alpha-1} (w+k)^2 |\nabla \varphi|^2\,dx
+C\int_{B(x_0,2r)\cap\partial\Omega} \tau v^{\alpha-1} (w+k) \varphi^2\,d\sigma.
\end{equation}
Let $U:= v^{\frac{\alpha-1}{2}} (w+k)$. Then, by $w+k\le v$, we find that
\begin{align*}
|\nabla (U\varphi)| &\le U|\nabla \varphi| + \varphi \dfrac{\alpha-1}{2} v^{\frac{\alpha-3}{2}}(w+k) |\nabla v|
+ \varphi v^{\frac{\alpha-1}{2}}|\nabla w| \\
&\le U|\nabla \varphi| +\varphi \dfrac{\alpha-1}{2} v^{\frac{\alpha-1}{2}} |\nabla v| + \varphi v^{\frac{\alpha-1}{2}}|\nabla w|,
\end{align*}
which, together with \eqref{eq4.7} and the assumption $w\ge0$, implies that
\begin{align}\label{eq4.8}
&\int_{B(x_0,2r)\cap\Omega} |\nabla (U\varphi)|^2\,dx\notag \\
&\quad\le C \int_{B(x_0,2r)\cap\Omega} \left[U^2 |\nabla \varphi|^2
+ (\alpha -1)^2 \varphi^2v^{\alpha-1}|\nabla v|^2 + \varphi^2 v^{\alpha-1} |\nabla w|^2\right]\,dx\notag \\
&\quad\le C(\alpha+1)\int_{B(x_0,2r)\cap\Omega}U^2|\nabla \varphi|^2\,dx +C\int_{B(x_0,2r)\cap\partial\Omega}
\tau  (w+k)^{-1} (U\varphi)^2\,d\sigma\notag \\
&\quad\le C(\alpha+1)  \int_{B(x_0,2r)\cap\Omega} U^2 |\nabla \varphi|^2\,dx + C_1k^{-1}
\int_{B(x_0,2r)\cap\partial\Omega} \tau  (U\varphi)^2\,d\sigma.
\end{align}
Moreover, from the assumption $\mathrm{supp\,}(\varphi)\subset B(x_0,2r)$, we deduce that
$\mathrm{Tr}(U\varphi)|_{\partial B(x_0,2r)\cap\Omega}=0$. By this and Lemma \ref{le2.5},
we conclude that
\begin{align}\label{eq4.9}
\|U\varphi \|_{W^{1,2}(B(x_0,2r)\cap\Omega)}\sim\|\nabla(U\varphi)\|_{L^2(B(x_0,2r)\cap\Omega)}.
\end{align}
From the assumption $q\in [s/(s+2-n),\infty]$ when $n\ge3$ and $q\in(1,\infty]$ when $n=2$, it follows that
$2q'\in[2,2s/(n-2)]$ when $n\ge 3$ and $2q'\in[2,\infty)$ when $n=2$. This, combined with Lemma \ref{le2.8}, Theorem \ref{th2.1}(iii),
$\mathrm{supp\,}(U\varphi)\subset B(x_0,2r)$, and \eqref{eq4.9}, further yields that
\begin{align*}
\|U\varphi\|_{L^{2q'}(B(x_0,2r)\cap\partial\Omega,\sigma)}&\le
\|U\varphi\|_{L^{2q'}(\partial T(x_0,2r)\cap\partial\Omega,\sigma)}\lesssim\|U\varphi\|_{W^{1,2}(T(x_0,2r))}\notag \\
&\sim\|U\varphi\|_{W^{1,2}(B(x_0,2r)\cap\Omega)}\sim\|\nabla(U\varphi)\|_{L^2(B(x_0,2r)\cap\Omega)},
\end{align*}
where $T(x_0,2r)$ denotes the tent domain as in Lemma \ref{le2.8},
which, together with H\"older's inequality, implies that
\begin{align}\label{eq4.10}
\int_{B(x_0,2r)\cap\partial\Omega}\tau(U\varphi)^2\,d\sigma
&\le\|\tau\|_{L^q(B(x_0,2r)\cap\partial\Omega,\sigma)}
\|U\varphi\|^2_{L^{2q'}(B(x_0,2r)\cap\partial\Omega,\sigma)}\notag \\
&\le C_2\|\tau\|_{L^q(B(x_0,2r)\cap\partial\Omega,\sigma)}\|\nabla(U\varphi)\|^2_{L^2(B(x_0,2r)\cap\Omega)}.
\end{align}
Thus, choosing $k:=2C_1C_2 \|\tau\|_{L^q(B(x_0,2r)\cap\partial\Omega,\sigma)}+\varepsilon$
with any given $\varepsilon\in(0,\infty)$ and applying \eqref{eq4.8} and \eqref{eq4.10}, we then obtain
\begin{equation}\label{eq4.11}
\int_{B(x_0,2r)\cap\Omega} |\nabla (U\varphi)|^2\,dx \le C(\alpha+1) \int_{B(x_0,2r)\cap\Omega}U^2 |\nabla \varphi|^2\,dx.
\end{equation}
Using Poincar\'e's inequality as in Lemma \ref{le2.5} and \eqref{eq4.11},
we conclude that there exists $p_2\in(2,2n/(n-2)]$ when $n\ge3$ and $p_2\in(2,\infty)$ when $n=2$ such that
\begin{align}\label{eq4.12}
\left[\fint_{B(x_0,2r)\cap\Omega} |U\varphi|^{p_2}\,dx\right]^{\frac{2}{p_2}}
&\le Cr^2\fint_{B(x_0,2r)\cap\Omega}|\nabla (U\varphi)|^2\,dx\notag \\
&\le C(\alpha+1)r^2
\fint_{B(x_0,2r)\cap\Omega} U^2 |\nabla \varphi |^2\,dx,
\end{align}
where $C$ is a positive constant independent of $\alpha$, $r$, and $U\varphi$.
		
Next, we continue as in the standard Moser iteration argument. Given any two positive constants $t_1$ and $t_2$ with $r\le t_1<t_2<2r$, we
take $0\le \varphi\le 1$ such that $\varphi\equiv1$ on $B(x_0,t_1)$, $\mathrm{supp\,}(\varphi)\subset B(x_0,t_2)$, and
$|\nabla \varphi|\lesssim 1/(t_2-t_1)$.
Then, by \eqref{eq4.12}, we find that
\begin{equation}\label{eq4.13}
\left[\fint_{B(x_0,t_1)\cap\Omega}|U|^{p_2}\,dx\right]^{\frac2{p_2}}\le
C\frac{(\alpha+1)r^2}{(t_2-t_1)^2}\fint_{B(x_0,t_2)\cap\Omega}|U|^{2}\,dx,
\end{equation}
where $C$ is a positive constant independent of $\alpha$, $r$,
$t_1$, $t_2$, and $U$.
Letting $M\to +\infty$, then $U$ tends to $(w+k)^{(\alpha+1)/2}$ and,
applying the monotone convergence theorem in \eqref{eq4.13},
we further conclude that
\begin{equation}\label{eq4.14}
\left[\fint_{B(x_0,t_1)\cap\Omega} |w+k|^{\frac{p_2(\alpha+1)}{2}}\,dx \right]^{\frac2{p_2}}\le C\frac{(\alpha+1)r^2}{(t_2-t_1)^2}
\fint_{B(x_0,t_2)\cap\Omega}|w+k|^{\alpha+1}\,dx.
\end{equation}
		
Let $\gamma:=p_2/2>1$ and, for any $m\in\mathbb{Z}_+$,
$r_m:= (1+ 2^{-m})r$ and $\alpha_m:= 2\gamma^m-1$. Applying \eqref{eq4.14}
with $t_1:= r_{m+1}$, $t_2:=r_{m}$, and $\alpha:=\alpha_m$, we then
find that, for any $m\in\mathbb{Z}_+$,
\begin{equation}\label{eq4.15}
\left[\fint_{B(x_0,r_{m+1})\cap\Omega}|w+k|^{\frac{p_2(\alpha_m+1)}2} \,dx\right]^{\frac2{p_2}}\le C(\alpha_m+1)2^{2m}
\left[\fint_{B(x_0,r_{m})\cap\Omega}|w+k|^{\alpha_m+1}\,dx\right].
\end{equation}
For any $m\in\mathbb{Z}_+$, let
$${\rm I}_m :=\left[\fint_{B(x_0,r_m)\cap\Omega}|w+k|^{(\alpha_m+1)}\,dx \right]^{\frac1{\alpha_m+1}}.$$
From $\alpha_m= 2\gamma^m-1$ with $\gamma=p_2/2$, we infer that, for any $m\in\mathbb{Z}_+$, $p_2(\alpha_m+1)/2=\alpha_{m+1}+1$,
which, combined with \eqref{eq4.15}, further implies that
\begin{equation}\label{eq4.16}
{\rm I}_{m+1}^{\alpha_m+1}\le C(\alpha_m+1)
2^{2m} {\rm I}_m^{\alpha_m+1}.
\end{equation}
Notice that ${\rm I_0}=[\fint_{B(x_0,2r)\cap\Omega}|w+k|^{2}
\,dx]^{\frac12}<\infty$. By this and \eqref{eq4.16}, we conclude that,
for any $m\in\mathbb{N}$, ${\rm I_m}<\infty$. Moreover, applying \eqref{eq4.16}
again yields, for any $m\in\mathbb{Z}_+$,
\begin{equation}\label{eq4.17}
\log({\rm I}_{m+1})\le\log({\rm I}_m)+\frac{1}{\alpha_m+1}[\log(C)+\log(\alpha_m+1)
+2m\log(2)].
\end{equation}
Meanwhile, from $\alpha_m+1=2(p_2/2)^m$ for any $m\in\mathbb{Z}_+$, it follows that
$$
\sum_{m=0}^\infty\frac{1}{\alpha_m+1}[\log(C)+\log(\alpha_m+1)
+2m\log(2)]\lesssim1,
$$
which, together with \eqref{eq4.17}, further implies that
$$\limsup\limits_{m\to \infty}\log({\rm I}_m)\le\log({\rm I}_0)+C$$
and hence, by $k=2C_1C_2\|\tau\|_{L^q(B(x_0,2r)\cap\partial\Omega,\sigma)}+\varepsilon$,
we have
\begin{align*}
\|w+k\|_{L^{\infty}(B(x_0,r)\cap\Omega)} &\le \limsup\limits_{m\to \infty}
{\rm I}_m \lesssim{\rm I}_0\\
&\lesssim\left[\fint_{B(x_0,2r)\cap\Omega}|w+k|^{2}\,dx\right]^{\frac12}
\lesssim\left[\fint_{B(x_0,2r)\cap\Omega}|w|^{2}\,dx\right]^{\frac12}+k\\
&\sim\left[\fint_{B(x_0,2r)\cap\Omega}|w|^{2}\,dx \right]^{\frac12}+
\|\tau\|_{L^q(B(x_0,2r)\cap\partial\Omega,\sigma)}+\varepsilon.
\end{align*}
Using this, the Minkowski norm inequality of $\|\cdot\|_{L^{\infty}(B(x_0,r)\cap\Omega)}$,  and $k=2C_1C_2\|\tau\|_{L^q(B(x_0,2r)\cap\partial\Omega,\sigma)}+\varepsilon$ 
again and then letting $\varepsilon\to0$,
we conclude that \eqref{eq4.2} holds.
This then finishes the proof of Lemma \ref{le4.1}.
\end{proof}

Applying the approach used in the proof of Lemma \ref{le4.1}, we can also obtain the following Moser estimate for weak solutions to
the local Robin problem; we omit the details here.

\begin{corollary}\label{c4.1}
Let $n\ge 2,\ s\in (n-2,n)$, and $(\Omega,\sigma)$ satisfy Assumption \ref{Ass} with $\Omega$ being bounded.
Assume that $\beta$ satisfies \eqref{eq1.1}, $K_0\in [1,\infty)$ is as in Lemma \ref{le2.8}, $x_0\in\partial\Omega$,
and $r\in(0,\mathrm{diam\,}(\Omega)/4]$.
Let $u\in W^{1,2}(B(x_0,2K^2r)\cap\Omega)$ be a weak solution to the Robin problem
\begin{equation*}
\begin{cases}
-\operatorname{div} (A\nabla u) = 0&\text{in}~~B(x_0,2K_0r),\\
A\nabla u\cdot \boldsymbol{\nu} +\beta u =0&\text{on}~~B(x_0,2K_0r)\cap\partial\Omega.
\end{cases}
\end{equation*}
Then there exists a positive constant $C$ depending only on $n$, the geometric constants for $(\Omega,\sigma)$, and $\mu_0$ in \eqref{eq1.2} such that
\begin{equation*}
\sup\limits_{x\in B(x_0,r)\cap\Omega}|u(x)|\le C\left[\fint_{B(x_0,2r)\cap\Omega} |u|^2\,dx\right]^{\frac12}.
\end{equation*}
\end{corollary}
	
\begin{lemma}\label{le4.2}
Let $n\ge 2,\ s\in (n-2,n)$, and $(\Omega,\sigma)$ satisfy Assumption \ref{Ass} with $\Omega$ being bounded.
Assume that the constant $K_0\in [1,\infty)$
is as in Lemma \ref{le2.8}, $x_0\in\partial\Omega$, $r\in(0,\mathrm{diam\,}(\Omega)/(4K_0)]$, and the tent domain $T(x_0,r)$
is the same as in Lemma \ref{le2.8}. Let the non-negative function
$\tau\in L^q(\Gamma(x_0,2r),\sigma)$ for some $q\in[s/(s+2-n),\infty]$ when $n\ge3$ and some $q\in(1,\infty]$
when $n=2$, where $\Gamma(x_0,2r):=\partial T(x_0,2r)\cap\partial\Omega$.
Assume that $w\in W^{1,2}(T(x_0,2K_0r))$ is bounded, ${\rm Tr}(u)=0$ on $S(x_0,2r):=\partial T(x_0,2r)\cap\Omega$, and $w$
satisfies that, for any $\varphi \in W^{1,2}(T(x_0,2r))$ satisfying
$\varphi\ge 0$ and ${\rm Tr}(\varphi)=0$ on $S(x_0,2r)$,
\begin{equation}\label{eq4.18}
\int_{T(x_0,2r)} A\nabla w\cdot \nabla \varphi\,dx
\le \int_{\Gamma(x_0,2r)}\tau \varphi\,d\sigma.
\end{equation}
Then
\begin{equation}\label{eq4.19}
\sup\limits_{x\in B(x_0,r)\cap\Omega}u(x) \le C\|\tau\|_{L^q(\Gamma(x_0,2r),\sigma)},
\end{equation}
where the positive constant $C$ depends only on the geometric
constants for $(\Omega,\sigma)$ and $\mu_0$ in \eqref{eq1.2}.
\end{lemma}
	
\begin{proof}
We first consider the case that $u\ge0$. In this case,
applying the assumption \eqref{eq4.18}, we then find that
\begin{equation}\label{eq4.20}
\int_{T(x_0,2r)}|\nabla u|^2\,dx \lesssim \int_{T(x_0,2r)}
A\nabla u \cdot \nabla u\,dx\lesssim\int_{\Gamma(x_0,2r)}\tau u\,d\sigma.
\end{equation}	
Moreover, by H\"older's inequality, Theorem \ref{th2.1}(iii),
assumption ${\rm Tr}(u)=0$ on $S(x_0,2r)$, and Poincar\'e's inequality as in Lemma \ref{le2.5} (see also \cite[Theorem 7.1]{DFM23}),
we conclude that
\begin{align*}
\int_{\Gamma(x_0,2r)} \tau u\,d\sigma&\le\|\tau\|_{L^q(\Gamma(x_0,2r),\sigma)}
\left[\int_{\Gamma(x_0,2r)}|u|^{q'}\,d\sigma\right]^{\frac1{q'}}\notag \\
&\lesssim \|\tau\|_{L^q(\Gamma(x_0,2r),\sigma)}  \|u\|_{W^{1,2}(T(x_0,2r))}\notag \\
& \lesssim\|\tau\|_{L^q(\Gamma(x_0,2r),\sigma)}\left[\int_{T(x_0,2r)}
|\nabla u|^2\,dx\right]^{\frac12},
\end{align*}
which, combined with \eqref{eq4.20}, implies that		
\begin{equation}\label{eq4.21}
\left[\int_{T(x_0,2r)}|\nabla u|^2\,dx\right]^{\frac12}\lesssim \|\tau\|_{L^q(\Gamma(x_0,2r),\sigma)}.
\end{equation}
Furthermore, from Lemma \ref{le4.1}, we infer that
\begin{equation}\label{eq4.22}
\sup\limits_{x\in B(x_0,r)\cap\Omega} u(x) \lesssim
\left[\fint_{T(x_0,2r)} u^2\,dx\right]^{\frac12}+ \|\tau\|_{L^q(\Gamma(x_0,2r),\sigma)}.
\end{equation}
Let $\overline{u}:=\fint_{T(x_0,2r)\cap B(x_0,2r)}u\,dx$. Then, applying Poincar\'e's inequality in Lemma \ref{le2.3} with
$\Omega:=T(x_0,2r)$, $p:=2$, and $E:=T(x_0,2r)\cap B(x_0,2r)$, we conclude that
\begin{equation}\label{eq4.23}
\fint_{T(x_0,2r)} u^2\,dx \lesssim\overline{u}^2 + \fint_{T(x_0,2r)} |u-\overline{u}|^2\,dx \lesssim\overline{u}^2 +\fint_{T(x_0,2r)}
|\nabla u|^2\,dx.
\end{equation}
Meanwhile, from Lemma \ref{le2.4}, we further deduce that
\begin{equation}\label{eq4.24}
\overline{u}^2 =\fint_{S(x_0,2r)}|\mathrm{Tr} (u) -
\overline{u}|^2\,d\sigma_\ast
\lesssim\fint_{\partial T(x_0,2r)}
|{\rm Tr}(u)-\overline{u}|^2\,d\sigma_\ast
\lesssim\fint_{T(x_0,2r)} |\nabla u|^2\,dx.
\end{equation}
Thus, by \eqref{eq4.21}, \eqref{eq4.23}, and \eqref{eq4.24},
we find that
$$ \fint_{T(x_0,2r)} u^2\,d\sigma\lesssim\fint_{T(x_0,2r)}|\nabla u|^2\,dx
\lesssim\|\tau\|^2_{L^q(\Gamma(x_0,2r),\sigma)},$$
which, combined with \eqref{eq4.22}, further implies that \eqref{eq4.19}
holds in the case that $u\ge0$.

Next, we consider the general case for $u$. Applying Lemma \ref{le2.7}
to $(T(x_0,2r),\sigma_\ast)$, we conclude that there exists a function $v\in W^{1,2}(T(x_0,2r))$
satisfying $\mathrm{Tr}(v)= 0$ on $S(x_0,2r)$ such that, for any $\varphi \in W^{1,2}(\Omega)$ with $\varphi \ge 0$
on $T(x_0,2r)$ and $\mathrm{Tr}(\varphi) = 0$ on $S(x_0,2r)$,
\begin{equation}\label{eq4.25}
\int_{T(x_0,2r)} A\nabla v \cdot \nabla \varphi\,dx = \int_{\Gamma(x_0,2r)} \tau \varphi\,d\sigma.
\end{equation}
By taking $\varphi= v^-:= \max\{-v,0\}$ in \eqref{eq4.25} and using the non-negativity
of $\tau$, we further conclude that $v$ is non-negative on $T(x_0,2r)$.
Therefore, estimate \eqref{eq4.19} holds for $v$.
Let $w := v-u$.	From the assumption that $u$ and $v$ respectively satisfy \eqref{eq4.18} and \eqref{eq4.25}, we deduce that,
for any $\varphi \in W^{1,2}(\Omega)$ with $\varphi \ge 0$
on $T(x_0,2r)$ and $\mathrm{Tr}(\varphi) = 0$ on $S(x_0,2r)$,
$$ \int_{T(x_0,2r)} A\nabla w\cdot \nabla \varphi\,dx \ge 0. $$
Taking $\varphi= w^-:=\max\{-w,0\}$, we then find that
$$ 0\le \int_{T(x_0,2r)} A\nabla w\cdot \nabla w^-\,dx = -\int_{T(x_0,2r)} A\nabla w^- \cdot \nabla w^-\,dx\le0,$$
which implies that $\nabla w^{-}=0$ on $T(x_0,2r)$. From this and $w^-=0$ on $S(x_0,2r)$,
we infer that $w^-=0$ on $T(x_0,2r)$. Thus, $w\ge 0$ on $T(x_0,2r)$, that is, $u\le v$ on $T(x_0,2r)$.
By this and the fact that \eqref{eq4.19} also holds for $v$, we conclude that \eqref{eq4.19} also holds for $u$.
This finishes the proof of Lemma \ref{le4.2}.
\end{proof}

\begin{lemma}\label{le4.3}
Let $n\ge 2,\ s\in (n-2,n)$, and $(\Omega,\sigma)$ satisfy Assumption \ref{Ass} with $\Omega$ being bounded.
Assume that $x_0\in\partial\Omega$ and $r\in(0,\mathrm{diam\,}(\Omega))$.
Then there exist constants $K_1\in(1,\infty)$ depending only on the geometric constants for $(\Omega,\sigma)$ and
$C\in(0,\infty)$ depending only on the geometric constants for $(\Omega,\sigma)$ and $\mu_0$ in \eqref{eq1.2} such that,
if $0\le u\in W^{1,2}(B(x_0,K_1r)\cap\Omega)$ and, for any $\varphi \in W^{1,2}(B(x_0,K_1r)\cap\Omega)$ satisfying
$\varphi\ge 0$ and $\varphi\equiv0$ on $\Omega\backslash B(x_0,\rho)$ for some $\rho\in(0,K_1r)$,
\begin{equation*}
\int_{B(x_0,K_1r)\cap\Omega} A\nabla u\cdot \nabla \varphi\,dx\le0,
\end{equation*}
then, for any $y\in B(x_0,r/4)\cap\Omega$, $u(y)\le Cu(y_{x_0}).$
Here $y_{x_0}$ denotes a corkscrew point with respect to $\Omega\cap B(x_0,r/4)$.
\end{lemma}

Lemma \ref{le4.3} is precisely \cite[Lemma 3.3]{DDEMM24}.

Moreover, we also need the following boundary Harnack inequality for the local Neumann problem established
in \cite[Theorem 3.1]{DDEMM24}.

\begin{lemma}\label{le4.4}
Let $n\ge 2,\ s\in (n-2,n)$, and $(\Omega,\sigma)$ satisfy Assumption \ref{Ass} with $\Omega$ being bounded.
Assume that $x_0\in\partial\Omega$ and $r\in(0,\mathrm{diam\,}(\Omega))$.
Then there exist constants $K_2\in[1,\infty)$ depending only on the geometric constants for $(\Omega,\sigma)$ and
$\theta\in(0,1)$ depending only on the geometric constants for $(\Omega,\sigma)$ and $\mu_0$ in \eqref{eq1.2} such that,
if $0\le u\in W^{1,2}(B(x_0,K_2r)\cap\Omega)$ and, for any $\varphi \in W^{1,2}(B(x_0,K_2r)\cap\Omega)$ satisfying
$\varphi\ge 0$ and $\varphi\equiv0$ on $\Omega\backslash B(x_0,\rho)$ for some $\rho\in(0,K_2r)$,
\begin{equation*}
\int_{B(x_0,K_2r)\cap\Omega} A\nabla u\cdot \nabla \varphi\,dx=0,
\end{equation*}
then
\begin{equation*}
\inf_{x\in B(x_0,r)\cap\Omega}u(x)\ge \theta \sup_{x\in B(x_0,r)\cap\Omega}u(x).
\end{equation*}
\end{lemma}

Now, we prove Theorem \ref{th1.2} by using Lemmas \ref{le4.1}, \ref{le4.2}, \ref{le4.3}, and \ref{le4.4}.
	
\begin{proof}[Proof of Theorem \ref{th1.2}]
We show Theorem \ref{th1.2} by borrowing some ideas from the proof of \cite[Theorem 4.4]{DDEMM24}.
Let $K\in(1,\infty)$ be such that $K\ge\max\{K_0,K_1,K_2\}$, where the constants $K_0$, $K_1$, and $K_2$ are respectively
as in Lemmas \ref{le2.8}, \ref{le4.3}, and \ref{le4.4}, $x_0\in\partial\Omega$, and $r\in(0,\mathrm{diam\,}(\Omega)/(4K)]$.
Assume that $u\in W^{1,2}(B(x_0,2K^2r)\cap\Omega)$ is a non-negative weak solution to the local Robin problem \eqref{eq1.6}.
Applying Lemma \ref{le2.7}, we conclude that there exists $h\in W^{1,2}(T(x_0,2Kr))$ satisfying
$\mathrm{Tr}(h)=0$ on $S(x_0,2Kr)$ such that, for any $\varphi \in W^{1,2}(T(x_0,2Kr))$ with $\varphi \ge 0$ on $T(x_0,2Kr)$
and $\mathrm{Tr}(\varphi) = 0$ on $S(x_0,2Kr)$,
\begin{equation}\label{eq4.26}
\int_{T(x_0,2Kr)} A\nabla h\cdot \nabla \varphi\,dx =-\int_{\Gamma(x_0,2Kr)}\beta u\varphi\,d\sigma.
\end{equation}
Then, by taking $\tau\equiv0$ in Lemma \ref{le4.2} and using the fact that $-\beta u\le0$ on $\Gamma(x_0,2Kr)$
and Lemma \ref{le4.2}, we find that
\begin{align}\label{eq4.27}
\sup\limits_{x\in B(x_0,r)\cap\Omega} h(x) \le0.
\end{align}
		
Moreover, from \eqref{eq4.26}, it follows that $-h$ satisfies
\begin{equation*}
\int_{T(x_0,2Kr)} A\nabla (-h)\cdot \nabla \varphi\,dx = \int_{\Gamma (x_0,2Kr)}\beta u \varphi\,d\sigma.
\end{equation*}
Similarly, by taking $\tau:=\beta u$ in Lemma \ref{le4.2} and applying Lemma \ref{le4.2}
 again, we conclude that
\begin{align}\label{eq4.28}
\sup\limits_{x\in B(x_0,r)\cap\Omega}(-h)(x)\le C\|\beta u\|_{L^{q_0}(\Gamma(x_0,2r),\sigma)}.
\end{align}
From \eqref{eq1.4} and \eqref{eq1.7}, we deduce that
$$\|\beta u\|_{L^{q_0}(\Gamma(x_0,2r),\sigma)}\le \|\beta\|_{L^{q_0}(\Gamma (x_0,2r),\sigma)}
\|\mathrm{Tr}(u)\|_{L^\infty(\Gamma(x_0,2r),\sigma)}
\le c_0 \sup\limits_{x\in B(x_0,2Kr)\cap\Omega} u(x),$$
which, combined with \eqref{eq4.28}, implies that
\begin{align*}
\sup\limits_{x\in B(x_0,r)\cap\Omega} (-h)(x)\le C c_0 \sup\limits_{x\in B(x_0,2Kr)\cap\Omega} u(x).
\end{align*}
This, together with \eqref{eq4.27}, further yields that there exists a positive constant $C_3$ such that
\begin{align}\label{eq4.29}
\|h\|_{L^{\infty}(B(x_0,r)\cap\Omega)}\le C_3c_0\sup\limits_{x\in B(x_0,2Kr)\cap\Omega} u(x).
\end{align}
		
Finally, notice that $u-h\in W^{1,2}(T(x_0,2Kr))$ satisfies
$$ \int_{T(x_0,2Kr)} A\nabla (u-h)\cdot \nabla \varphi \,dx = 0. $$
Thus, applying Lemma \ref{le4.4} to $w:= u-h +\|h\|_{L^{\infty}(B(x_0,r)\cap\Omega)}$,
we then find that there exists a positive constant $\theta \in (0,1)$ such that
\begin{align}\label{eq4.30}
\inf\limits_{x\in B(x_0,r)\cap\Omega} w(x)\ge\theta\sup\limits_{x\in B(x_0,r)\cap\Omega} w(x).
\end{align}
Moreover, from Lemma \ref{le4.3}, Assumption (A3), and the interior Harnack inequality (see, for instance, \cite[Theorem 8.20]{GT83}),
it follows that there exists a positive constant $C_4\in(0,1]$ depending only on the geometric constants for
$(\Omega,\sigma)$ such that
$$\sup\limits_{x\in B(x_0,r)\cap\Omega}u(x)\ge C_4\sup\limits_{x\in B(x_0,2Kr)\cap\Omega} u(x),$$
which, combined with \eqref{eq4.29} and \eqref{eq4.30}, implies that
\begin{align}\label{eq4.31}
\inf\limits_{x\in B(x_0,r)\cap\Omega}u(x) &\ge \inf\limits_{x\in B(x_0,r)\cap\Omega} w(x) -2\|h\|_{L^{\infty}(B(x_0,r)\cap\Omega)} \notag \\
&\ge\theta \sup\limits_{x\in B(x_0,r)\cap\Omega} w(x) -2\|h\|_{L^{\infty}(B(x_0,r)\cap\Omega)}\notag \\
&\ge C_4\theta \sup\limits_{x\in B(x_0,2Kr)\cap\Omega} u(x) -2\|h\|_{L^{\infty}(B(x_0,r)\cap\Omega)} \notag \\
&\ge (C_4\theta -C_3c_0)\sup\limits_{x\in B(x_0,2Kr)\cap\Omega} u(x)\ge(C_4\theta -C_3c_0) \sup\limits_{x\in B(x_0,r)\cap\Omega} u(x).
\end{align}
Take $c_0$ small enough such that $C_3c_0 =C_4\theta/2$ in \eqref{eq4.31} and let $\eta:=C_4\theta/2\in(0,1)$. Then,
by \eqref{eq4.31}, we conclude that
\begin{align*}
\inf\limits_{x\in B(x_0,r)\cap\Omega}u(x)\ge\eta\sup\limits_{x\in B(x_0,r)\cap\Omega} u(x).
\end{align*}	
This finishes the proof of \eqref{eq1.8} and hence Theorem \ref{th1.2}.
\end{proof}
	
\section{Proof of Theorem \ref{th1.3}\label{S4}}
	
In this section, we prove the existence and regularity estimates of Green's functions for the Robin problem \eqref{Robin-Pro}
and further give the proof of Theorem \ref{th1.3}.
	
\begin{proof}[Proof of Theorem \ref{th1.3}]
Let $g\in L^2(\Omega)$. A function $v\in W^{1,2}(\Omega)$ is called a \emph{weak solution} to the Robin problem
\begin{equation}\label{eq5.1}
\begin{cases}
-\mathrm{div}(A\nabla v) = g~~&\text{in}~\Omega,\\
A\nabla v\cdot \boldsymbol{\nu} + \beta v=0~~&\text{on}~\partial \Omega
\end{cases}
\end{equation}
if, for any $\varphi\in C_{\rm c}^{\infty}(\mathbb{R}^n)$,
\begin{equation*}
\int_{\Omega} A\nabla v \cdot \nabla \varphi\,dx + \int_{\partial \Omega}\beta \mathrm{Tr}(v)\varphi\,d\sigma = \int_{\Omega}g\varphi\,dx.
\end{equation*}
Applying Lemma \ref{le2.1}, Theorem \ref{th2.1}, and the Lax--Milgram theorem
and repeating the proof of Theorem \ref{th2.2}, we conclude that there exists a unique weak solution $v\in W^{1,2}(\Omega)$
to the Robin problem \eqref{eq5.1} and
\begin{equation}\label{eq5.2}
\|v\|_{W^{1,2}(\Omega)} \le C \|g\|_{L^2(\Omega)}
\end{equation}
with $C$ being a positive constant depending only on $n$, the geometric constants for $(\Omega,\sigma)$, $E_0$ in \eqref{eq1.1}, and $\mu_0$ in \eqref{eq1.2}.	
		
Then, using \eqref{eq5.2} and following the standard approach to construct Green's function,
developed by Gr\"uter and Widman in the proof of \cite[Theorem (1.1)]{gw82} (see also the proof of \cite[Theorem 5.6]{DDEMM24}),
we obtain the existence of Green's function $G_R$ for the Robin problem \eqref{Robin-Pro}; we omit the details here.
		
Moreover, repeating the proof of \cite[(5.4)]{DDEMM24}, we then find that the representation formula
\eqref{eq1.10} for the weak solution to the Robin problem \eqref{Robin-Pro} holds.
		
Next, we prove the regularity estimates \eqref{eq1.11} and \eqref{eq1.12}.
Let $R_0:= \mathrm{diam\,}(\Omega)$. By checking the proof of \cite[Theorem 3.3(i)]{CK14} carefully,
we find that the conclusion of \cite[Theorem 3.3]{CK14} remain holds when the domain $\Omega$ satisfies Assumption \ref{Ass}.
Thus, from \cite[Theorems 3.2 and 3.3]{CK14}, we deduce that, for any $x,y\in\Omega$ and any $t\in(0,\infty)$,
\begin{equation}\label{eq5.3}
|p_{t,L_R}(x,y)|\lesssim \frac{1}{\min\{t^{n/2}, R_0^n\}}\exp\left(-\frac{|x-y|^2}{ct}\right),
\end{equation}
where $\{p_{t,L_R}\}_{t\in(0,\infty)}$ denote the heat kernels associated with the Robin problem \eqref{eq5.1}
and $c\in(0,\infty)$ is a positive constant depending only on $\mu_0$ in \eqref{eq1.2}.
Moreover, by the well-known relation between the heat kernels $\{p_{t,L_R}\}_{t\in(0,\infty)}$ and the Green function
$G_R$ (see, for instance, \cite[Theorem 2.12 and (3.11)]{DK09}), we find that, for any $x,y\in\Omega$ with $x\neq y$,
$$ G_R(x,y)= \int_{0}^{\infty} p_{t,L_R}(x,y)\,dt. $$
Thus, for any $x,y\in\Omega$ with $x\neq y$,
\begin{equation}\label{eq5.4}
0\le G_R(x,y) \le \int_{0}^{|x-y|^2} + \int_{|x-y|^2}^{2R_0^2} + \int_{2R_0^2}^{\infty} |p_{t,L_R}(x,y)|\,dt=:I_1+I_2
+I_3.
\end{equation}
For $I_1$, using \eqref{eq5.3}, we obtain
\begin{align}\label{eq5.5}
I_1 &\lesssim\int_{0}^{|x-y|^2}\frac{1}{t^{n/2}}\exp\left(-\frac{|x-y|^2}{ct}\right)\,dt
\lesssim\int_{0}^{|x-y|^2}\frac{1}{|x-y|^n}\,dt \lesssim \frac{1}{|x-y|^{n-2}}.
\end{align}
For $I_2$, by \eqref{eq5.3} again, we conclude that
\begin{align}\label{eq5.6}
I_2\lesssim\int_{|x-y|^2}^{2R_0^2} t^{-n/2}\,dt
\lesssim \begin{cases}
1+\log\left(\frac{R_0}{|x-y|}\right)&\text{when}~n=2,\\
|x-y|^{2-n}&\text{when}~n\ge 3.
\end{cases}
\end{align}
Moreover, from \cite[(5.48)]{CK14} (see also \cite[(3.59)]{DK09}), we deduce that,
when $t\in[2R_0^2,\infty)$, for any $x,y\in\Omega$,
$$
|p_{t,L_R}(x,y)|\lesssim R_0^{-n}e^{-c(t-2R_0^2)},
$$
which further implies that
$$ I_3 \lesssim\int_{2R_0^2}^{\infty} R_0^{-n} e^{-c(t-2R_0^2)}\,dt\lesssim R_0^{-n}\lesssim|x-y|^{2-n}. $$
By this, \eqref{eq5.4}, \eqref{eq5.5}, and \eqref{eq5.6}, we conclude that, for any $x,y\in\Omega$ with $x\neq y$,
$$ 0\le G_R(x,y) \lesssim \begin{cases}
1 + \log\left(\dfrac{1}{|x-y|}\right)&\text{if }n=2,\\
\frac{1}{|x-y|^{n-2}}&\text{if }n\ge 3,
\end{cases} $$
which proves \eqref{eq1.11}.
		
Now, we show the H\"older regularity of the Green function $G_R$ when $n\ge 3$
by considering three cases.
Let $K_0\in[1,\infty)$ be as in Lemma \ref{le2.8}.
		
\emph{Case (1)} $2|x_1-x_2| < |x_1-y|\le 8K_0|x_1-x_2|$. In this case, we have
$ |x_1-x_2| \sim |x_1-y| \sim |x_2-y|.$
Thus, from this and \eqref{eq1.11}, we infer that, for any given $\delta\in (0,1]$,
\begin{align}\label{eq5.7}
|G_R(x_1,y) - G_R(x_2,y)| &\le G_R(x_1,y) + G_R(x_2,y) \lesssim |x_1-y|^{2-n}\notag \\
&\lesssim \dfrac{|x_1-x_2|^{\delta}}{|x_1-y|^{n-2+\delta}}.
\end{align}
		
\emph{Case (2)} $8K_0|x_1-x_2| <|x_1-y|$ and $|x_1-y|/(4K_0)\le\delta(x_1)$,
where $\delta(y):=\mathrm{dist\,}(y,\partial\Omega)$. In this case,
it is easy to find $x_2\in B(x_1,|x_1-y|/(5K_0))\subset\Omega$ and,
for any $x\in B(x_1,|x_1-y|/(5K_0))$, $|x-y|\sim|x_1-y|$.
By this, \cite[Lemma 11.30]{DFM23}, and \eqref{eq1.11}, we conclude that there exists $\delta\in (0,1]$ such that
\begin{align}\label{eq5.8}
|G_R(x_1,y)-G_R(x_2,y)|&\le \mathop{\mathrm{osc}\,}_{B(x_1,|x_1-x_2|)}
G_R(\cdot,y) \notag \\
&\lesssim \dfrac{|x_1-x_2|^{\delta}}{|x_1-y|^{\delta}}\left[\fint_{B(x_1,|x_1-y|/ (5K_0))}|G_R(x,y)|^2\,dx \right]^{\frac12}\notag \\
&\lesssim\dfrac{|x_1-x_2|^{\delta}}{|x_1-y|^{\delta}}\frac{1}{|x_1-y|^{n-2}}
\sim\dfrac{|x_1-x_2|^{\delta}}{|x_1-y|^{n-2+\delta}}.
\end{align}

\emph{Case (3)} $8K_0|x_1-x_2| <|x_1-y|$ and $|x_1-y|/(4K_0)>\delta(x_1)$.
In this case, let $R:=|x_1-y|/(2K_0)$. Then $\delta(x_1)<R/2$ and $|x_1-x_2|<R/4$. Take $x_0\in\partial\Omega$ such that $|x_0-x_1|=\delta(x_1)$.
Then
\begin{align*}
|x_0-y|\ge|x_1-y|-|x_0-x_1|\ge2K_0R-\frac{R}{4}\ge\frac{7K_0}{4}R,
\end{align*}
which implies that $y\in\Omega\backslash B(x_0,7K_0R/4)$.
Meanwhile, it is easy to find that $x_1,x_2\in B(x_0,3R/4)\subset T(x_0,3R/2)$,
where the tent domain $T(x_0,R)$ is as in Lemma \ref{le2.8}.
Thus, $y\not\in T(x_0,3R/2)$ and $G_R(\cdot,y)$ satisfies
\begin{equation}\label{eq5.9}
\begin{cases}
-\operatorname{div} (A\nabla G_R(\cdot,y)) = 0&\text{in}~~T(x_0,3R/2),\\
A\nabla G_R(\cdot,y)\cdot \boldsymbol{\nu} +\beta G_R(\cdot,y) =0&\text{on}~~T(x_0,3R/2)\cap\partial\Omega.
\end{cases}
\end{equation}
Applying Lemma \ref{le3.2} and Corollary \ref{c4.1} to \eqref{eq5.9},
we further conclude that there exists $\delta\in (0,1]$ such that
\begin{align}\label{eq5.10}
|G_R(x_1,y)-G_R(x_2,y)|&\lesssim\dfrac{|x_1-x_2|^{\delta}}{R^{\delta}}
\left[\fint_{B(x_0,3R/2)}|G_R(x,y)|^2\,dx\right]^{\frac12}\notag \\
&\lesssim\dfrac{|x_1-x_2|^{\delta}}{|x_1-y|^{\delta}}
\frac{1}{|x_1-y|^{n-2}}\sim\dfrac{|x_1-x_2|^{\delta}}{|x_1-y|^{n-2+\delta}}.
\end{align}
Then, from \eqref{eq5.7}, \eqref{eq5.8}, and \eqref{eq5.10}, it follows that \eqref{eq1.12} holds.
This then finishes the proof of Theorem \ref{th1.3}.
\end{proof}
	
\section{Proofs of Theorems \ref{th1.4} and \ref{th1.5}\label{S6}}
In this section, we prove the existence of the harmonic measure for the Robin problem \eqref{Robin-Pro}
and its absolute continuity with respect to the surface measure, by using
the Riesz representation theorem and both the maximum principle and the boundary Harnack inequality
for the Robin problem \eqref{Robin-Pro}.
	
We begin with establishing the maximum principle for the Robin problem \eqref{Robin-Pro}.
\begin{lemma}\label{le6.1}
Let $n\ge 2$, $s\in (n-2,n)$, and $(\Omega,\sigma)$ satisfy Assumption \ref{Ass} with $\Omega$ being bounded. Assume that the
function $\beta$ is as in \eqref{eq1.1}. Then, for any $f\in
C(\partial\Omega)$, there exists a unique $u\in W^{1,2}(\Omega) \cap C(\overline{\Omega})$
such that, for any $\varphi \in C_{\rm c}^{\infty}(\mathbb{R}^n)$,
\begin{equation*}
\int_{\Omega} A\nabla u \cdot \nabla \varphi\,dx + \int_{\partial \Omega}\beta
\mathrm{Tr}(u)~\varphi\,d\sigma = \int_{\partial \Omega} f\varphi\,d\sigma.
\end{equation*}
\end{lemma}
\begin{proof}
The existence and uniqueness of the weak solution $u\in W^{1,2}(\Omega)$ can be deduced from
Theorem \ref{th2.2}. Moreover, by Theorem \ref{th1.1},
we further find that $u\in C(\overline{\Omega})$. This finishes the proof of
Lemma \ref{le6.1}.
\end{proof}
	
\begin{lemma}\label{le6.2}
Let $n\ge 2$, $s\in (n-2,n)$, and $(\Omega,\sigma)$ satisfy Assumption \ref{Ass} with $\Omega$ being bounded, and
let $\beta$ be the same as in \eqref{eq1.1} and satisfy $\beta\ge a_0$ on $\partial\Omega$ with $a_0$ being a given positive constant.
\begin{itemize}
\item[{\rm (i)}] Assume that $u$ is the weak solution to the Robin problem \eqref{Robin-Pro} with $0\le f\in L^p(\partial\Omega,\sigma)$
for some $p\in(s/(s+2-n),\infty]$.
Then $u\ge 0$ in $\Omega$.
\item[{\rm (ii)}]
Let $u$ be the weak solution to the Robin problem \eqref{Robin-Pro} with $f\in C(\partial\Omega)$. Then, for any $x\in \Omega$,
\begin{equation}\label{eq6.1}
|u(x)|\le\dfrac{1}{a_0}\max\limits_{z\in\partial\Omega}|f(z)|.
\end{equation}
\end{itemize}
\end{lemma}
\begin{proof}
We first prove (i). Let $u^- := \max\{-u,0\}$. Since $u^-$ is the maximum function of two Sobolev
functions in $W^{1,2}(\Omega)$, it follows that $u^-\in W^{1,2}(\Omega)$.
To show $u\ge0$ in $\Omega$, it suffices to prove $u^- = 0$ in $\Omega$.
Applying the density of $C_{\rm c}^{\infty}(\mathbb{R}^n)$ in $W^{1,2}(\Omega)$
(see, for instance, \cite[Theorem 2.1]{DDEMM24}), we conclude that
$$ \int_{\Omega} A\nabla u \cdot \nabla u^-\,dx + \int_{\partial \Omega}\beta
\mathrm{Tr}(u)\mathrm{Tr}(u^-)\,d\sigma = \int_{\partial \Omega} fu^-\,d\sigma. $$
Obviously, the left-hand side of the above equality is less than or equal to zero
and the right-hand side is more than or equal to zero. Thus, $\nabla u^- = 0$
in $\Omega$ and $u^- = 0$ on $\partial \Omega$, which implies $u^- = 0$ in $\Omega$.
Therefore, $u\ge0$ in $\Omega$, which completes the proof of (i).
		
Now, we show (ii). Let $M:=\max_{z\in\partial\Omega}|f(z)|$. From the assumption that
$u$ is the weak solution to the Robin problem \eqref{Robin-Pro}, it follows that, for any $\varphi\in C^\infty_{\rm c}(\mathbb{R}^n)$,
$$\int_{\Omega} A\nabla u \cdot \nabla\varphi\,dx + \int_{\partial \Omega}\beta
\mathrm{Tr}(u)\varphi\,d\sigma = \int_{\partial \Omega} f\varphi\,d\sigma,
$$
which implies that
$$\int_{\Omega} A\nabla\left(\frac{M}{a_0}-u\right) \cdot \nabla\varphi\,dx + \int_{\partial \Omega}\beta
\left[\frac{M}{a_0}-\mathrm{Tr}(u)\right]\varphi\,d\sigma = \int_{\partial \Omega}\left(\frac{M}{a_0}\beta- f\right)\varphi\,d\sigma.
$$
Thus, $M/a_0-u$ is the weak solution of the Robin problem
\begin{equation*}
\begin{cases}
-\mathrm{div}(A\nabla v) = 0~~&\text{in}~\Omega,\\
A\nabla v\cdot \boldsymbol{\nu}+\beta v =\frac{M\beta}{a_0}-f~~&\text{on}~\partial \Omega.
\end{cases}
\end{equation*}
By the assumptions that $\beta$ satisfies \eqref{eq1.1}, $\beta\ge a_0$ on $\partial\Omega$,
and $f\in C(\partial \Omega)$, we find that
$$0\le M\beta/a_0-f\in L^p(\partial\Omega,\sigma)$$
with some $p\in(s/(s+2-n),\infty]$.
Therefore, from this and (i), we deduce that $M/a_0-u\ge0$ in $\Omega$. Similarly,
we also have $M/a_0+u\ge 0$. Therefore, \eqref{eq6.1} holds.
This finishes the proof of (ii) and hence Lemma \ref{le6.2}.
\end{proof}
	
Next, we prove the existence of harmonic measures (see, for instance, \cite[Definition 1.2.6]{K94})
associated with the Robin problem \eqref{Robin-Pro}
by using the maximum principle in Lemma \ref{le6.2} and the Riesz representation theorem
(see, for instance, \cite[Theorem 1.38]{EG15}).
	
\begin{proposition}\label{pro6.1}
Let $n\ge 2$, $s\in (n-2,n)$, and $(\Omega, \sigma)$ satisfy Assumption \ref{Ass} with $\Omega$ being bounded.
Assume that $\beta$ satisfies \eqref{eq1.1} and $\beta\ge a_0$ on $\partial\Omega$
with $a_0$ being a given positive constant.
Then there exists a family of Radon measures $\{w_{R}^x\}_{x\in\Omega}$ supported
in $\partial \Omega$ such that, for any $f\in C(\partial\Omega)$, the unique weak solution $u_f$
 to the Robin problem \eqref{Robin-Pro} with data $f$ is given by, for any $x\in\Omega$,
\begin{equation}\label{eq6.2}
u_f(x)=\int_{\partial \Omega} f(z)\,dw_{R}^x(z).
\end{equation}
\end{proposition}
	
\begin{proof}
By Theorem \ref{th2.2}, we find that, for any given $f \in C(\partial \Omega)$, there exists a unique weak
solution $u_f$ to the Robin problem \eqref{Robin-Pro} with data $f$.
		
Moreover, the linearity of the map $f \mapsto u_f$ follows immediately from the
definition of weak solutions to the Robin problem \eqref{Robin-Pro} and its boundedness follows
from Lemma \ref{le6.2}(ii).
		
For any $x\in \Omega$, define the map $f\mapsto u_f(x)$. Since it is linear and
bounded with norm $1/a_0$ from $C(\partial \Omega)$ to $\mathbb{R}$, it follows from
the Riesz representation theorem (see, for instance, \cite[Theorem 1.38]{EG15})
that there exists a family of Radon measures $\{w_{R}^x\}_{x\in\Omega}$ on $\partial\Omega$
such that equation \eqref{eq6.2} holds. This finishes the proof of Proposition \ref{pro6.1}.
\end{proof}

\begin{proposition}\label{pro6.2}
Let $n\ge 2$, $s\in (n-2,n)$, and $(\Omega, \sigma)$ satisfy Assumption \ref{Ass} with $\Omega$ being bounded.
Assume that $\beta$ satisfies \eqref{eq1.1} and $\beta\ge a_0$ on $\partial\Omega$
with $a_0$ being a given positive constant.
Then, for any Borel measurable set $E\subset\partial\Omega$,
the function $u(x):= w^x_R(E)$ with $x\in\Omega$ satisfies that, for any $\varphi\in C_{\rm c}^{\infty}(\mathbb{R}^n)$,
\begin{equation}\label{eq6.3}
\int_{\Omega} A\nabla u\cdot\nabla\varphi\,dx+\int_{\partial\Omega}\beta
\mathrm{Tr}(u)\varphi\,d\sigma=\int_{E} \varphi\,d\sigma.
\end{equation}
That is, $u$ is the weak solution to the Robin problem \eqref{Robin-Pro} with data $f:={\bf 1}_E$.
\end{proposition}
	
\begin{proof}
We prove this proposition by following the approach used in the proof of \cite[Theorem 5.5]{DDEMM24}.
Fix $x_0\in\Omega$. From Proposition \ref{pro6.1}, we deduce that $w_R^{x_0}$ is a Radon measure supported in
$\partial \Omega$. Moreover,  notice that $\sigma$ is also a Radon measure on $\partial\Omega$.
Therefore, for each $i \in \mathbb{N}$, there exist a compact set $K_i$ and an open set $O_i$ in $\partial\Omega$
such that $K_i\subset E\subset O_i$, $\omega^{x_0}_R(O_i\backslash K_i)< 1/i$, and
$\sigma(O_i\backslash K_i) < 1/i$.
		
For any given $i \in \mathbb{N}$, let $f_i \in C(\partial \Omega)\cap L^2(\partial \Omega,\sigma)$ be such that
$0 \le f_i \le 1$, $f_i \equiv 1$ on $K_i$, and $\mathrm{supp}(f_i) \subset O_i$.
For any $x\in\Omega$, define
$$
u_i(x):= \int_{\partial \Omega} f_i(z)\, dw_R^x(z).
$$
Then, by Proposition \ref{pro3.1} and Lemma \ref{le6.2}(ii), we conclude that,
for any $i\in\mathbb{N}$, $u_i\in C(\overline{\Omega})$ and $\sup_{y\in\overline{\Omega}}|u_i(y)|\le 1/a_0$.
Moreover, from the definition of $\{u_i\}_{i\in\mathbb{N}}$, we deduce that the sequence $\{u_i(x)\}_{i\in\mathbb{N}}$ converges for any
given $x \in\Omega$. Meanwhile, by Proposition \ref{pro6.1}, we find that, for any $i\in\mathbb{N}$,
$u_i$ is the weak solution to the Robin problem \eqref{Robin-Pro} with data $f_i$, which, combined with Proposition \ref{pro3.1}
and the choice that $\sigma(O_i\backslash K_i) < 1/i$ for any $i\in\mathbb{N}$, implies that,
for any $i,j\in\mathbb{N}$, any given compact subset $K$ of $\Omega$, and any $y,z\in K$,
\begin{align*}
|u_i(y)-u_j(y)| &\le |u_i(z)-u_j(z)| + C_{(K)}\|f_i-f_j\|_{L^p(\partial\Omega,\sigma)}\\
&\le|u_i(z)-u_j(z)|+ C_{(K)}\max\left\{i^{-1/p},j^{-1/p}\right\},
\end{align*}
where $p\in (s/(s+2-n),\infty)$
and $C_{(K)}$ is a positive constant depending on $K$ but independent of $i$ and $j$.
Thus, for any given compact subset $K$ of $\Omega$, $\{u_i\}_{i\in\mathbb{N}}$ converges uniformly on
$K$ to a function $u$ which weakly solves $-\mathrm{div}(A \nabla u) = 0$ in $\Omega$.
		
Next, we show that, for any given $x_0\in\Omega$, $u(x_0) = w_R^{x_0}(E)$.	
Indeed, for any $\varepsilon\in(0,\infty)$, there exists
$i_0\in\mathbb{N}$ large enough such that $1/i_0<\varepsilon$, which, together with the choice of the sets
$\{K_i\}_{i\in\mathbb{N}}$ and $\{O_i\}_{i\in\mathbb{N}}$, implies that
when $i \ge i_0$, $w_R^{x_0}(O_i\setminus K_i) < 1/i<\varepsilon$.
From this and the definition of $u_i$, we deduce that, when $i \ge i_0$,
$$|u_i(x_0)-w^{x_0}_R(E)| \le w_R^{x_0}(O_i\setminus K_i)<\varepsilon,$$
which further implies that $u(x_0)=w_R^{x_0}(E)$.
		
Finally, we prove that $u$ satisfies \eqref{eq6.3}.
By the definition of $u_i$, we find that, for any $\varphi\in C_{\rm c}^\infty(\mathbb R^n)$,
\begin{equation}\label{eq6.4}
\int_{\Omega} A\nabla u_i\cdot \nabla \varphi\,dx +\int_{\partial \Omega}\beta
\mathrm{Tr}(u_i)\varphi\,d\sigma
=\int_{\partial \Omega} f_i\varphi\,d\sigma.
\end{equation}
From Theorem \ref{th2.2} and the construction of $\{f_i\}_{i\in\mathbb{N}}$, we infer that, for any $i\in\mathbb{N}$,
$$
\|u_i\|_{W^{1,2}(\Omega)} \le C\|f_i\|_{L^2(\partial \Omega,\sigma)}\le  C \sigma(E),
$$
which implies that $\{u_i\}_{i\in\mathbb{N}}$ is weakly convergent in $W^{1,2}(\Omega)$ and hence
$\{u_i\}_{i\in\mathbb{N}}$ weakly converges to $u$ in $W^{1,2}(\Omega)$. By this and Theorem \ref{th2.1}(iii),
we further conclude that $\{\mathrm{Tr}(u_i)\}_{i\in\mathbb{N}}$ weakly converges to $\mathrm{Tr}(u)$ in $L^2(\partial \Omega,\sigma)$.
Thus, letting $i\to\infty$ in \eqref{eq6.4}, we then find that $u$ satisfies \eqref{eq6.3}, which completes the
proof of Proposition \ref{pro6.2}.
\end{proof}
	
Now, we prove Theorems \ref{th1.4} and \ref{th1.5}
by using Propositions \ref{pro6.1} and \ref{pro6.2} and Theorems
\ref{th1.2} and \ref{th1.3}.
	
\begin{proof}[Proof of Theorem \ref{th1.4}]
Notice that both the existence of harmonic measures $\{w_{R}^x \}_{x\in \Omega}$
associated with the Robin problem \eqref{Robin-Pro} and
the expression \eqref{eq1.13} follow from Proposition
\ref{pro6.1}.

Now, we prove \eqref{eq1.14} and \eqref{eq1.15}. By \eqref{eq1.10} and Proposition \ref{pro6.2}, we find that, for any
$x\in \Omega$ and any Borel measurable set $E\subset \partial \Omega$,
\begin{equation}\label{eq6.5}
w_R^x(E) = \int_{E} G_R(y,x)\,d\sigma(y),
\end{equation}
which, together with Theorem \ref{th1.3}, yields that $\sigma(E)=0$ implies $w_R^x(E) = 0$.
Thus, for any given $x\in \Omega$, $\sigma \ll w_R^x$.
Moreover, for any given $x\in\Omega$, if $w_R^x(E)= 0$ for the Borel measurable set $E\subset \partial \Omega$,
from \eqref{eq6.5} and the continuity and the non-negativity of the Green function $G_R(y,x)$ on the variable $y$,
it follows that either $G_R(y,x) = 0$ for any $y \in E$ or $\sigma(E) = 0$. If $G_R(y,x) = 0$, then, by Theorem
\ref{th1.2}, we find that $G_R(y,x) = 0$ for all $y \in \overline{\Omega}$
sufficiently close to $E$. In this case, applying Theorem \ref{th1.2} again, we find that
$G_R(\cdot,y) = 0$ on $\Omega$. However, if we take $\phi = 1$ in
\eqref{eq1.9}, this leads to a contradiction. Thus,
$\sigma(E)= 0$, which further yields that $w_R^x(E) = 0$ implies $\sigma(E)=0$, that is,  $w_R^x\ll\sigma$.
Therefore, for any $x\in\Omega$, $w_R^x \ll \sigma \ll w_R^x$
and hence \eqref{eq1.14} holds. Furthermore, the expression \eqref{eq1.15}
directly follows from \eqref{eq6.5}.
This then finishes the proof of Theorem \ref{th1.4}.
\end{proof}
	
\begin{proof}[Proof of Theorem \ref{th1.5}]
Let $x_0\in\partial\Omega$ and $r\in(0,\mathrm{diam\,}(\Omega)/(4K)]$ satisfy $\|\beta\|_{L^{q_0}(B(x_0,2Kr)\cap\partial\Omega)}\le c_0$,
where the constants $K\in(1,\infty)$ and $c_0\in(0,1)$ are as in Theorem \ref{th1.2}.
Assume further that $x\in\Omega\backslash B(x_0,2K^2r)$.
From \eqref{eq6.5}, it follows that, for any measurable
set $E\subset B(x_0,r)\cap\partial\Omega$,
$$
\sigma(E)\inf_{y \in B(x_0,r)\cap\partial\Omega} G_R(y,x)\le w_R^x(E)\le
\sigma(E)\sup_{y \in B(x_0,r)\cap\partial\Omega} G_R(y,x)
$$
and
\begin{align*}
\sigma(B(x_0,r)\cap\partial\Omega)\sup_{y\in B(x_0,r)\cap \partial\Omega}
G_R(y,x)&\ge w_R^x(B(x_0,r)\cap\partial\Omega) \\
&\ge\sigma(B(x_0,r)\cap\partial\Omega)
\inf_{y \in B(x_0,r)\cap\partial\Omega} G_R(y,x),
\end{align*}
which implies that, for any measurable set $E\subset B(x_0,r)\cap \partial \Omega$,
\begin{align}\label{eq6.6}
&\frac{\inf_{y \in B(x_0,r)\cap\partial\Omega} G_R(y,x)}{\sup_{y \in B(x_0,r)
\cap\partial\Omega} G_R(y,x)}\frac{\sigma(E)}{\sigma(B(x_0,r)\cap\partial\Omega)}\notag \\
&\quad\le \frac{w_R^x(E)}{w_R^x(B(x_0,r)\cap\partial\Omega)}\le \frac{\sup_{y\in B(x_0,r)\cap\partial\Omega
} G_R(y,x)}{\inf_{y\in B(x_0,r)\cap\partial\Omega} G_R(y,x)}
\frac{\sigma(E)}{\sigma(B(x_0,r)\cap\partial\Omega)}.
\end{align}
Based on \eqref{eq6.6}, to show \eqref{eq1.16},
it suffices to prove
\begin{align}\label{eq6.7}
\sup_{y \in B(x_0,r)\cap\partial\Omega}G_R(y,x) \lesssim \inf_{y \in B(x_0,r)
\cap\partial\Omega} G_R(y,x).
\end{align}
Since $G_R(\cdot,x)$ is continuous up to the boundary (see Theorem
\ref{th1.3}), it follows that \eqref{eq6.7} can be deduced
from a more general version
\begin{equation}\label{eq6.8}
\sup_{y \in B(x_0,r)\cap\Omega} G_R(y,x) \lesssim \inf_{y \in B(x_0,r)\cap\Omega} G_R(y,x).
\end{equation}
By the definition of the Green function $G_R$, we find that $G_R(x,\cdot)$ is a weak solution to the local Robin problem
\eqref{eq1.6}. Thus, from Theorem \ref{th1.2},
we deduce that \eqref{eq6.8} holds.
This finishes the proof of Theorem \ref{th1.5}.
\end{proof}
	
	
	

\bigskip
	
\noindent Jiayi Wang and Sibei Yang
	
\medskip
	
\noindent School of Mathematics and Statistics, Gansu Key Laboratory of Applied
Mathematics and Complex Systems, Lanzhou University, Lanzhou 730000, The
People's Republic of China
	
\smallskip
	
\noindent {\it E-mails}: \texttt{wangjiayi2024@lzu.edu.cn} (J. Wang)
	
\hspace{0.888cm} \texttt{yangsb@lzu.edu.cn} (S. Yang)
	
\bigskip
	
\noindent Dachun Yang (Corresponding author)
	
\medskip
	
\noindent Laboratory of Mathematics and Complex Systems (Ministry of Education
of China), School of Mathematical Sciences, Beijing Normal University, Beijing 100875,
The People's Republic of China
	
\smallskip
	
\noindent{\it E-mail:} \texttt{dcyang@bnu.edu.cn}	
	
\end{document}